\documentclass[dvipdfmx,a4paper,11pt]{article}
\usepackage[a4paper,top=30truemm,bottom=30truemm,left=25truemm,right=25truemm]{geometry}
\usepackage{amsmath}
\usepackage{url}
\usepackage{amsthm}
\usepackage{amssymb}
\usepackage{yhmath}
\usepackage{relsize}
\usepackage{cmll}
\usepackage{bussproofs}
\usepackage{framed}
\usepackage{ascmac}
\usepackage{fancybox}
\usepackage{amssymb}
\usepackage{mathrsfs}
\usepackage{latexsym}
\usepackage{lipsum}
\usepackage{titlesec}
\usepackage{tikz}
\usepackage[titletoc,title]{appendix}
\numberwithin{equation}{section} 
\usetikzlibrary{positioning, snakes}
\usepackage{setspace}
\newtheoremstyle{mystyle}
{}
{}
{\itshape}{}
{\bfseries}{.}
{4pt}{}

\theoremstyle{mystyle}

\newtheorem{Def}{Definition}[section]

\newtheorem{Thm}[Def]{Theorem}
\newtheorem{Lem}[Def]{Lemma}

\newtheorem{Cor}[Def]{Corollary}

\let\oldenumerate\enumerate
\renewcommand{\enumerate}{
\oldenumerate
\setlength{\itemsep}{1pt}
\setlength{\parskip}{2pt}
\setlength{\parsep}{0pt}
}
\let\olditemize\itemize
\renewcommand{\itemize}{
\olditemize
\setlength{\itemsep}{1pt}
\setlength{\parskip}{2pt}
\setlength{\parsep}{0pt}
}
\let\olddescription\description
\renewcommand{\description}{
\olddescription
\setlength{\itemsep}{1pt}
\setlength{\parskip}{2pt}
\setlength{\parsep}{0pt}
}
\newenvironment{bprooftree}
  {\leavevmode\hbox\bgroup}
  {\DisplayProof\egroup}

    \EnableBpAbbreviations
\newcommand{\negr}[1]{\,\sim\!\!#1}
\newcommand{\negl}[1]{-#1}
\newcommand{\negrl}[1]{\sim\!\!-#1} 
\newcommand{\cdott}[1]{\,\,\cdot^{\star}\!#1}  
\newcommand{\cdottt}[1]{\,\,\cdot\!#1} 
\newcommand{\negrr}[1]{\sim\sim\!\!#1} 
\newcommand{\neglr}[1]{-\!\!\sim\!#1} 
 
\newcommand{\darrowa}[1]{\,\downarrow_{\mathbf{A}}\!\!#1} \newcommand{\darrowgx}[1]{\,\downarrow_{G_X}\!\!#1} 
\newcommand{\uarrowgx}[1]{\uparrow_{G_X}\!\!#1} 
\newcommand{\Crit}[1]{\text{Crit}#1}
\newcommand{\dom}[1]{\mathop{\mathrm{dom}}\nolimits#1}
\newcommand{\N}[1]{\,\,N\,\,#1}

\newcommand{\R}[1]{\,\,R\,\,#1}

\newcommand{\Nh}[1]{\,\,N_h\,\,#1}
\newcommand{\Nb}[1]{\,\,N_B\,\,#1}

\newcommand{\Nz}[1]{\,\,N_{B^0}\,\,#1}
\newcommand{\id}{\mathop{\mathrm{id}}\nolimits}

\newcommand{\dbackslash}[1]{\backslash\!\!\backslash#1}
\newcommand{\dslash}[1]{/\!\!/#1}
\newcommand{\ifff}[1]{\quad \text{iff} \quad#1}
\newtheorem{LEM}{Lemma}

\begin{document}
\title{Decision Problems for Propositional Non-associative Linear Logic and Extensions}
\setcounter{footnote}{1}
\author{Hiromi Tanaka \thanks{Graduate School of Letters, Keio University, Tokyo, Japan.
		E-mail: \texttt{hiromitanaka@keio.jp}}
	}
\date{}
\maketitle
\begin{abstract}
\noindent 
In our previous work, we proposed the logic obtained from full non-associative Lambek calculus by adding a sort of linear-logical modality. 
We call this logic non-associative non-commutative intuitionistic linear logic ($\mathbf{NACILL}$, for short).  
In this paper, we establish the decidability and undecidability results for various extensions of $\mathbf{NACILL}$. 
Regarding the decidability results, we show that the deducibility problems for several extensions of $\mathbf{NACILL}$ with the rule of left-weakening are decidable. 
Regarding the undecidability results, we show that the provability problems for all the extensions of non-associative non-commutative classical linear logic by the rules of contraction and exchange are undecidable.
\end{abstract}
\section{Introduction} 
In the early period of linear logic, Lincoln-Mitchell-Scedrov-Shankar \cite{LMSS92} proved that the provability problems for propositional linear logic and propositional non-commutative linear logic are both undecidable.
In view of this result, it is natural to ask how the lack of associativity of multiplication affects the decision problems for linear logic and related systems. 
So far, however, it has hardly been investigated whether the decision problems for non-associative versions of propositional linear logic are decidable or not, whereas several substructural logicians investigated the decision problems for various non-associative logics, see e.g., \cite{BvA04,Bus16,Bus17,Ch15,Fa08,GJ13,HH14}.

Under these circumstances, in order to develop the work by Lincoln et al. in a non-associative setting, the author \cite{Tan19} proposed propositional non-associative non-commutative intuitionistic linear logic (denoted by $\mathbf{NACILL}$) and showed that all the extensions $\mathbf{NACILL}$ by the rules of contraction and exchange are undecidable. 
In this paper, as a continuation of our previous work, we advance further the program of settling the decision problems for non-associative versions of propositional linear logic. 
Our contribution is summarized as follows:
\begin{enumerate}
	\item Contrary to the undecidability results established in \cite{Tan19}, we show that the deducibility problems for several extensions of $\mathbf{NACILL}$ with the rule of left-weakening are decidable. 
	\item Also, we show that the provability problem for any of the extensions of propositional non-associative non-commutative classical linear logic ($\mathbf{NACCLL}^-$, for short) by the rules of contraction and exchange is undecidable. 
\end{enumerate}

For our first contribution, we show that various subvarieties of integral interior residuated lattice-ordered unital groupoids (integral interior $r \ell u$-groupoids, for short) have the finite embeddability property (FEP, for short). 
Here, an interior $r \ell u$-groupoid is just a residuated lattice-ordered unital groupoid ($r \ell u$-groupoid, for short) equipped with a conucleus as a fundamental operation. 
Actually, some subvarieties of interior $r \ell u$-groupoids form equivalent algebraic semantics for extensions of $\mathbf{NACILL}$. 
Our proof of the FEP for integral interior $r \ell u$-groupoids relies heavily on the techniques by Blok-van Alten \cite{BvA04} and Galatos-Jipsen \cite{GJ13}. 
Blok-van Alten showed that various subclasses of integral residuated partially-ordered groupoids possess the FEP. 
After that, Galatos-Jipsen showed that any of the subvarieties of integral $r \ell u$-groupoids axiomatized by equations consisting only of operation symbols from the language $\{\cdot,\lor,1\}$ has the FEP, using relational models for (non-associative) substructural logics, called residuated frames. 
As well as in these two approaches, the assumption of integrality is crucial in our setting. 
To apply their techniques to our proof, we introduce slightly extended versions of residuated frames, called enriched residuated frames. 

For our second contribution, we show that every $r \ell u$-groupoid is embeddable into the reduct of a cyclic bounded involutive $r \ell u$-groupoid, using the idea from Galatos-Raftery \cite{GR04}. 
Immediately, it turns out that involutive full non-associative Lambek calculus (denoted by $\mathbf{InFNL}$) is strongly conservative over full non-associative Lambek calculus (denoted by $\mathbf{FNL}$).  
In conjunction with the undecidability result established by Chvalovsk\'{y} \cite{Ch15}, it follows that the deducibility problem for  $\mathbf{InFNL}$ is undecidable. 
Moreover, using the idea in \cite{Tan19}, we show that the provability problem for $\mathbf{NACCLL}^-$ is undecidable. 
We stress that Buszkowski \cite{Bus16} proved that the deducibility problem for $\mathbf{InFNL}$ is undecidable. 
However, the argument just described works well even when the rules of contraction and exchange are also concerned; consequently, the provability problems for all the extensions of $\mathbf{NACCLL}^-$ by contraction and exchange are undecidable.

We end the introduction by summarizing the contents of the following sections. 
Section~\ref{pre} consists of three parts. 
In the first part, we outline sequent calculi for $\mathbf{NACILL}$.  
In the second part, we introduce interior $r \ell u$-groupoids and confirm that each of the extensions of $\mathbf{NACILL}$ by the rules of weakening, contraction, and exchange is strongly complete with respect to a variety of interior $r \ell u$-groupoids.
In the third part, we recall some well-known notions in residuated structures, such as nuclei.  
Section \ref{main} consists of two parts.
In the first half, we show some basic properties of enriched residuated frames.
The second half is devoted to the proof of the FEP for integral interior $r \ell u$-groupoids.
In Section \ref{classical}, we prove that the provability problems for all the extension of $\mathbf{NACCLL}^-$ by the rules of contraction and exchange are undecidable.

\section{Preliminaries}
\label{pre}
\begin{figure}[t]
	\begin{description}
		\item[Initial sequents:]
		\[
		\begin{bprooftree}
		\AxiomC{}
		\RightLabel{(Id)}
		\UnaryInfC{$a \Rightarrow a$}
		\end{bprooftree}
		\begin{bprooftree}
		\AxiomC{}
		\RightLabel{($\Rightarrow 1$)}
		\UnaryInfC{$\varepsilon \Rightarrow 1$}
		\end{bprooftree}
		\begin{bprooftree}
		\AxiomC{}
		\RightLabel{$(0\Rightarrow)$}
		\UnaryInfC{$0 \Rightarrow \epsilon$}
		\end{bprooftree}
		\]
		\item[Cut:]
		\[
		\begin{bprooftree}
		\AxiomC{$x \Rightarrow a$}
		\AxiomC{$u(a) \Rightarrow \delta$}
		\RightLabel{$(cut)$}
		\BinaryInfC{$u(x) \Rightarrow \delta$}
		\end{bprooftree}
		\]
		\item[Rules for logical connectives:]
		\[
		\begin{bprooftree}
		\AxiomC{$u(\varepsilon) \Rightarrow \delta$}
		\RightLabel{$(1\Rightarrow)$}
		\UnaryInfC{$u(1) \Rightarrow \delta$}
		\end{bprooftree}
		\begin{bprooftree}
		\AxiomC{$x \Rightarrow \epsilon$}
		\RightLabel{$(\Rightarrow 0)$}
		\UnaryInfC{$x \Rightarrow 0$}
		\end{bprooftree}
		\begin{bprooftree}
		\AxiomC{$x \Rightarrow a$}
		\AxiomC{$u(b)\Rightarrow \delta$}
		\RightLabel{$(\backslash \Rightarrow)$}
		\BinaryInfC{$u(x \circ (a \backslash b)) \Rightarrow \delta$}
		\end{bprooftree}
		\]
		\[
		\begin{bprooftree}
		\AxiomC{$a \circ x \Rightarrow b$}
		\RightLabel{$(\Rightarrow \backslash)$}
		\UnaryInfC{$x \Rightarrow a \backslash b$}
		\end{bprooftree}
		\begin{bprooftree}
		\AxiomC{$u(a \circ b) \Rightarrow \delta$}
		\RightLabel{$(\cdot \Rightarrow)$}
		\UnaryInfC{$u(a \cdot b) \Rightarrow \delta$}
		\end{bprooftree}
		\begin{bprooftree}
		\AxiomC{$x \Rightarrow a$}
		\AxiomC{$y \Rightarrow b$}
		\RightLabel{$(\Rightarrow \cdot)$}
		\BinaryInfC{$x \circ y \Rightarrow a \cdot b$}
		\end{bprooftree}
		\]
		\[
		\begin{bprooftree}
		\AxiomC{$x \circ a \Rightarrow b$}
		\RightLabel{$(\Rightarrow /)$}
		\UnaryInfC{$x \Rightarrow b/a$}
		\end{bprooftree}
		\begin{bprooftree}
		\AxiomC{$x \Rightarrow a$}
		\AxiomC{$u(b) \Rightarrow \delta$}
		\RightLabel{$(/ \Rightarrow)$}
		\BinaryInfC{$u((b/a) \circ x) \Rightarrow \delta$}
		\end{bprooftree}
		\]
		\[
		\begin{bprooftree}
		\AxiomC{$u(a_i) \Rightarrow \delta$}
		\RightLabel{$(\land \Rightarrow)$ for $i=1,2$}
		\UnaryInfC{$u(a_1 \land a_2) \Rightarrow \delta$}
		\end{bprooftree}
		\begin{bprooftree}
		\AxiomC{$x \Rightarrow a$}
		\AxiomC{$x \Rightarrow b$}
		\RightLabel{$(\Rightarrow \land)$}
		\BinaryInfC{$x \Rightarrow a \land b$}
		\end{bprooftree}
		\]
		\[
		\begin{bprooftree}
		\AxiomC{$u(a) \Rightarrow\delta$}
		\AxiomC{$u(b) \Rightarrow \delta$}
		\RightLabel{$(\lor \Rightarrow)$}
		\BinaryInfC{$u(a \lor b) \Rightarrow \delta$}
		\end{bprooftree}
		\begin{bprooftree}
		\AxiomC{$x \Rightarrow a_i$}
		\RightLabel{$(\Rightarrow \lor)$ for $i=1,2$}
		\UnaryInfC{$x \Rightarrow a_1 \lor a_2$}
		\end{bprooftree}
		\]
			\[
		\begin{bprooftree}
		\AxiomC{$u(a) \Rightarrow \delta$}
		\RightLabel{$(\oc \Rightarrow)$}
		\UnaryInfC{$u(\oc a) \Rightarrow \delta$}
		\end{bprooftree}
		\begin{bprooftree}
		\AxiomC{$k \Rightarrow a$}
		\RightLabel{$(\Rightarrow \oc)$}
		\UnaryInfC{$k \Rightarrow \oc a$}
		\end{bprooftree}
		\begin{bprooftree}
		\AxiomC{$u(\varepsilon) \Rightarrow \delta$}
		\RightLabel{$(kw)$}
		\UnaryInfC{$u(k) \Rightarrow \delta$}
		\end{bprooftree}
		\]
		\[
		\begin{bprooftree}
		\AxiomC{$u(k \circ k) \Rightarrow \delta$}
		\RightLabel{$(kc)$}
		\UnaryInfC{$u(k) \Rightarrow \delta$}
		\end{bprooftree}
		\begin{bprooftree}
		\AxiomC{$u(k \circ y) \Rightarrow \delta$}
		\RightLabel{$(ke)$}
		\doubleLine
		\UnaryInfC{$u(y \circ k) \Rightarrow \delta$}
		\end{bprooftree}
		\]
		\[
		\begin{bprooftree}
		\AxiomC{$u((k \circ y) \circ z) \Rightarrow \delta$}
		\RightLabel{$(ka1)$}
		\doubleLine
		\UnaryInfC{$u(k \circ (y \circ z)) \Rightarrow \delta$}
		\end{bprooftree}
		\begin{bprooftree}
		\AxiomC{$u((x \circ y) \circ k) \Rightarrow \delta$}
		\RightLabel{$(ka2)$}
		\doubleLine
		\UnaryInfC{$u(x \circ (y \circ k)) \Rightarrow \delta$}
		\end{bprooftree}
		\]
	\end{description}
	\caption{Inference rules of $\mathbf{NACILL}^0$}
	\label{inf1}
\end{figure}

First of all, we describe a sequent calculus for \emph{propositional non-associative non-commutative intuitionistic linear logic with zero}, denoted by $\mathbf{NACILL}^0$. 
The \emph{language} $\mathcal{L}^0_{\oc}$ of $\mathbf{NACILL}^0$ consists of operation symbols $\cdot, \land, \lor, \backslash,/$ of arity $2$, $\oc$ of arity $1$, and $1,0$ of arity $0$. 
We fix the (countable) set of variables and denote it by $V$. 
An \emph{$\mathcal{L}^0_{\oc}$-formula} is just a term in the language $\mathcal{L}_{\oc}^0$ over $V$. 
In what follows, we write $Fm_{\mathcal{L}^0_{\oc}}$ for the set of $\mathcal{L}^0_{\oc}$-formulas, $(Fm_{\mathcal{L}^0_{\oc}}^{\circ},\circ,\varepsilon)$ for the free unital groupoid generated by $Fm_{\mathcal{L}^0_{\oc}}$, $U_{\mathcal{L}^0_{\oc}}$ for the set of unary linear polynomials over $Fm^{\circ}_{\mathcal{L}^0_{\oc}}$, and $K_{\mathcal{L}^0_{\oc}}$ for the free unital groupoid generated by the set $\{\oc a \mid a \in Fm_{\mathcal{L}^0_{\oc}}\}$. 
Given $u \in U_{\mathcal{L}^0_{\oc}}$ and $x \in Fm_{\mathcal{L}^0_{\oc}}^{\circ}$, $u(x)$ denotes the image of $x$ under $u$.  
Specifically, we write $\id$ for the identity polynomial, i.e., $\id(x)=x$ for any $x \in Fm_{\mathcal{L}^0_{\oc}}^{\circ}$. 
For a detailed discussion of unary linear polynomials, refer the reader to \cite{GO10}. 
An \emph{$\mathcal{L}^0_{\oc}$-sequent} is an element of the set $Fm_{\mathcal{L}^0_{\oc}}^\circ \times (Fm_{\mathcal{L}^0_{\oc}} \cup \{\epsilon\})$, where the symbol $\epsilon$ is called the \emph{empty stoup}. 
For the sake of readability, we always write $x \Rightarrow \delta$ for $(x,\delta) \in Fm_{\mathcal{L}^0_{\oc}}^\circ \times (Fm_{\mathcal{L}^0_{\oc}} \cup \{\epsilon\})$. 
A \emph{sequent calculus for $\mathbf{NACILL}^0$} consists of the inference rules displayed in Figure~\ref{inf1}. 
In Figure~\ref{inf1}, metavariables $x,y,z$ range over $Fm^{\circ}_{\mathcal{L}^0_{\oc}}$, $a,a_1,a_2,b$ over $Fm_{\mathcal{L}^0_{\oc}}$, $u$ over $U_{\mathcal{L}^0_{\oc}}$, $\delta$ over $Fm_{\mathcal{L}^0_{\oc}} \cup \{\epsilon\}$, and $k$ over $K_{\mathcal{L}^0_{\oc}}$. 
In particular, each of the rules equipped with a double line, i.e., ($ke$), ($ka1$), and ($ka2$), means that the lower sequent implies the upper sequent and vice versa.

Given a set $\mathcal{S} \cup \{s\}$ of $\mathcal{L}^0_{\oc}$-sequents, a \emph{proof} of $s$ in $\mathbf{NACILL}^0$ from $\mathcal{S}$ is inductively defined as follows: (i) an initial sequent $s$ is a proof of $s$ in $\mathbf{NACILL}^0$ from $\mathcal{S}$, (ii) a sequent $s$ from $\mathcal{S}$ is a proof of $s$ in $\mathbf{NACILL}^0$ from $\mathcal{S}$, and (iii) if $\Pi_1,\ldots,\Pi_n$ are proofs of $s_1,\ldots,s_n$ in $\mathbf{NACILL}^0$ from $\mathcal{S}$, respectively, and the expression of the form
\begin{prooftree}
	\AxiomC{$s_1$}
	\AxiomC{$\cdots$}
	\AxiomC{$s_n$}
	\RightLabel{$(r)$}
	\TrinaryInfC{$s$}
\end{prooftree}
is an instance of an inference rule $(r)$ in $\mathbf{NACILL}^0$, then the figure below is a proof of $s$ in $\mathbf{NACILL}^0$ from $\mathcal{S}$:
\begin{prooftree}
	\AxiomC{$\Pi_1$}
	\AxiomC{$\cdots$}
	\AxiomC{$\Pi_n$}
	\RightLabel{$(r)$}
	\TrinaryInfC{$s$}
\end{prooftree}
We say that a sequent $s$ is \emph{deducible} in $\mathbf{NACILL}^0$ from $\mathcal{S}$, and write $\mathcal{S} \vdash_{\mathbf{NACILL}^0} s$, if there exists a proof of $s$ in $\mathbf{NACILL}^0$ from $\mathcal{S}$. 
Specifically, we say that a sequent $s$ is \emph{provable} in $\mathbf{NACILL}^0$, and write $\vdash_{\mathbf{NACILL}^0} s$, if $s$ is deducible from the empty assumption. 

One can prove the cut-elimination for $\mathbf{NACILL}^0$. 
The verification is, however, beyond the scope of this paper; refer the interested reader to \cite[Appendix A]{Tan19} for a proof using enriched Gentzen frames. 
\begin{Thm}
	For any sequent $s$, if $s$ is provable in $\mathbf{NACILL}^0$, it is provable in $\mathbf{NACILL}^{0}$ without using the rule of $(cut)$.
\end{Thm} 

Given a set $R$ of inference rules, the \emph{extension} of $\mathbf{NACILL}^0$ by $R$ is the sequent calculus obtained from $\mathbf{NACILL}^0$ by adding all the rules from $R$, and is denoted by $\mathbf{NACILL}^0_R$.   
In an obvious way, we define the deducibility and the provability of sequents in $\mathbf{NACILL}^0_R$, for which we use the symbol $\vdash_{\mathbf{NACILL}^0_R}$. 
The most typical extensions of $\mathbf{NACILL}^0$ are obtained from $\mathbf{NACILL}^0$ by adding some of the following basic structural rules: 
\[
\begin{bprooftree}
\AxiomC{$u(x \circ y) \Rightarrow \delta$}
\RightLabel{$(e)$}
\UnaryInfC{$u(y \circ x) \Rightarrow \delta$}
\end{bprooftree}
\begin{bprooftree}
\AxiomC{$u(\varepsilon) \Rightarrow \delta$}
\RightLabel{$(i)$}
\UnaryInfC{$u(x) \Rightarrow \delta$}
\end{bprooftree}
\begin{bprooftree}
\AxiomC{$x \Rightarrow \epsilon$}
\RightLabel{$(o)$}
\UnaryInfC{$x \Rightarrow a$}
\end{bprooftree}
\begin{bprooftree}
\AxiomC{$u(x \circ x) \Rightarrow \delta$}
\RightLabel{$(c)$}
\UnaryInfC{$u(x) \Rightarrow \delta$}
\end{bprooftree}
\begin{bprooftree}
\AxiomC{$u((x \circ y) \circ z) \Rightarrow \delta$}
\RightLabel{$(a)$}
\doubleLine
\UnaryInfC{$u(x \circ (y \circ z)) \Rightarrow \delta$}
\end{bprooftree}
\]
For instance, $\mathbf{NACILL}^0_{ae}$ is just (the $\{\top,\bot\}$-free fragment of) Troelstra's \emph{intuitionistic linear logic with zero}, denoted by $\mathbf{ILZ}$ (cf. \cite{Tro92}). 
We often abbreviate the combination of the rules of $(i)$ and $(o)$ as $(w)$.

Next, we also review classical versions of $\mathbf{NACILL}^0$. 
\emph{Propositional non-associative non-commutative classical linear logic} (denoted by $\mathbf{NACCLL}^-$) is the extension of $\mathbf{NACILL}^0$ by the following three initial sequents:
\[
\begin{bprooftree}
\AxiomC{}
\RightLabel{(DNE1)}
\UnaryInfC{$\negrl a\Rightarrow a$}
\end{bprooftree}
\begin{bprooftree}
\AxiomC{}
\RightLabel{(DNE2)}
\UnaryInfC{$\neglr a\Rightarrow a$}
\end{bprooftree}
\begin{bprooftree}
\AxiomC{}
\RightLabel{(COMP)}
\UnaryInfC{$\negr a/b \Leftrightarrow a \backslash \!\!\negl \!b$}
\end{bprooftree}
\]
Here, we use the expression of the form $\negr a$ (resp. $\negl a$) to denote $a \backslash 0$ (resp. $0/a$), and the expression of the form $\negr a/b \Leftrightarrow a \backslash \!\!\negl \!b$ is an abbreviation of the sequents $\negr a/b \Rightarrow a \backslash \!\!\negl \!b$ and $a \backslash \!\!\negl \!b \Rightarrow \negr a/b$. 
It might be more natural that a classical version of $\mathbf{NACILL}^0$ has a cyclic negation. 
In view of this, one obtains another type of propositional non-associative non-commutative classical linear logic by adding the new initial sequent $\negr a \Leftrightarrow \negl a$ to $\mathbf{NACCLL}^-$. 
We denote this logic by $\mathbf{NACCLL}$. 

Let $\mathcal{M}$ be a sublanguage of $\mathcal{L}^0_{\oc}$, i.e., a subset of $\mathcal{L}^0_{\oc}$. 
The \emph{$\mathcal{M}$-fragment} (or \emph{$\mathcal{N}$-free fragment}) of $\mathbf{NACILL}^0$ (resp. $\mathbf{NACCLL}^-$, $\mathbf{NACCLL}$) is the sequent calculus obtained from $\mathbf{NACILL}^0$ (resp. $\mathbf{NACCLL}^-$, $\mathbf{NACCLL}$) by removing all the inference rules that are involved with the operation symbols from $\mathcal{N}$, where $\mathcal{N}$ is the complement of $\mathcal{M}$ in $\mathcal{L}^0_{\oc}$. 
For instance, we denote the $\mathcal{L}_{\oc}$-fragment of $\mathbf{NACILL}^0$ by  $\mathbf{NACILL}$, where $\mathcal{L}_{\oc}=\{\land,\lor,\cdot,\backslash,/,\oc,1\}$. 
Given a sublanguage $\mathcal{M}$ of $\mathcal{L}^0_{\oc}$ such that $0 \not\in \mathcal{M}$, we always assume that an $\mathcal{M}$-sequent is an element of the set $Fm_{\mathcal{M}}^\circ \times Fm_{\mathcal{M}}$, i.e., the right-hand side of every $\mathcal{M}$-sequent has exactly one formula. 
Moreover, we review some non-associative substructural logics.
Roughly speaking, non-associative substructural logics are just $\oc$-free fragments of extensions of $\mathbf{NACILL}$. 
For instance, \emph{full non-associative Lambek calculus} ($\mathbf{FNL}$) is the $\mathcal{L}$-fragment of $\mathbf{NACILL}^0$, where $\mathcal{L}=\{\land,\lor,\cdot,\backslash,/,1\}$. 
The logic $\mathbf{FNL}$ is equivalent to $\mathbf{GL}$ in \cite{GJ13} and to FNL1 in \cite{Bus17}. 
Likewise, in view of the constructions of $\mathbf{NACCLL}^-$ and $\mathbf{NACCLL}$, one also defines two classical versions of $\mathbf{FNL}$. 
\emph{Involutive full non-associative Lambek calculus} (resp. \emph{cyclic involutive full non-associative Lambek calculus}) is the $\mathcal{L}^{0}$-fragment of $\mathbf{NACCLL}^-$ (resp. $\mathbf{NACCLL}$) , where $\mathcal{L}^0=\{\land,\lor,\cdot,\backslash,/,1,0\}$, and is denoted by $\mathbf{InFNL}$ (resp. $\mathbf{CyInFNL}$). 
The logic $\mathbf{InFNL}$ is equivalent to $\mathbf{InGL}$ in \cite{GJ13} and to InFNL1 in \cite{Bus17}.

From now on, we discuss the algebraic models for $\mathbf{NACILL}^0$ and related systems. 
We briefly recall some basic notions from universal algebra, based on \cite{BvA04,Bur82,GJKO07,GO10}. 
Let $\mathcal{M}$ be a sublanguage of $\mathcal{L}^0_{\oc}$. 
A \emph{partial $\mathcal{M}$-algebra} is a structure of the form $\mathbf{A}=(A,(f^{\mathbf{A}})_{f \in \mathcal{M}})$ such that $A$ is a set and $f^{\mathbf{A}}$ is an $n$-ary partial operation on $A$ for each $f \in \mathcal{M}$, where $n$ stands for the arity of $f$. 
For each $f \in \mathcal{M}$, $f^{\mathbf{A}}$ is called a \emph{fundamental partial operation} of $\mathbf{A}$.
Given $x_1,\ldots,x_n \in A$ and an $n$-ary fundamental partial operation $f^{\mathbf{A}}$ of $\mathbf{A}$, we say that $f^{\mathbf{A}}(x_1,\ldots,x_n)$ is \emph{defined} if there exists $y \in A$ such that  $f^{\mathbf{A}}(x_1,\ldots,x_n)=y$. 
Otherwise, $f^{\mathbf{A}}(x_1,\ldots,x_n)$ is said to be \emph{undefined}.
For $f \in \mathcal{M}$ of arity $n$, we denote the set $\{(x_1,\ldots,x_n) \in A^n \mid \text{$f^{\mathbf{A}}(x_1,\ldots,x_n)$ is defined}\}$ by $\dom f^{\mathbf{A}}$. 
A partial $\mathcal{M}$-algebra  $\mathbf{A}$ is merely called an \emph{$\mathcal{M}$-algebra} if $f^{\mathbf{A}}$ is a total operation on $\mathbf{A}$ for any $f \in \mathcal{M}$.
The \emph{$\mathcal{N}$-reduct} of an $\mathcal{M}$-algebra $(A,(f^{\mathbf{A}})_{f \in \mathcal{M}})$ is an algebra $(A,(f^{\mathbf{A}})_{f \in \mathcal{N}})$, where $\mathcal{N}$ is a sublanguage of $\mathcal{M}$.
Given an $\mathcal{M}$-algebra $\mathbf{A}$, a \emph{partial subalgebra} of $\mathbf{A}$ is a partial $\mathcal{M}$-algebra $\mathbf{B}=(B,(f^{\mathbf{B}})_{f \in \mathcal{M}})$ such that $B \subseteq A$ and for each $n$-ary fundamental partial operation $f^{\mathbf{B}}$ and $x_1,\ldots,x_n \in B$:
\[
f^{\mathbf{B}}(x_1,\ldots,x_n)=
\begin{cases}
f^{\mathbf{A}}(x_1,\ldots,x_n) & \text{if $f^{\mathbf{A}}(x_1,\ldots,x_n) \in B$,}\\
\text{undefined} & \text{otherwise.}
\end{cases}
\]

In what follows, we recall various kinds of algebras in sublanguages of $\mathcal{L}^0_{\oc}$ in stages. 
\begin{Def}[see e.g., \cite{GJKO07, GO10}]
	{\normalfont A \emph{residuated lattice-ordered unital groupoid} (\emph{$r\ell u$-groupoid}, for short) is an algebra ${\mathbf{A}}=(A,\wedge,\vee,\cdot,\backslash,/,1)$ such that:
	\begin{itemize}
		\item $(A,\wedge,\vee)$ is a lattice,
		\item $(A,\cdot,1)$ is a unital groupoid, and
		\item for any $x,y,z \in A$, $x \cdot y \leq z$ iff $y \leq x\backslash z$ iff $x \leq z/y$.
	\end{itemize} }
\end{Def} 

An \emph{residuated lattice-ordered unital groupoid with zero} (\emph{$r \ell uz$-groupoid}, for short) is an algebra $(A,\wedge,\vee,\cdot,\backslash,/,1,0)$ such that $(A,\wedge,\vee,\cdot,\backslash,/,1)$ is an $r \ell u$-groupoid and $0 \in A$. 
We often abbreviate $x \backslash 0$ (resp. $0/x$) as $\negr a$ (resp. $\negl a$).
The class $\mathsf{RLUG}$ (resp. $\mathsf{RLUG}^0$) of $r\ell u$-groupoids (resp. $r\ell uz$-groupoids) forms a variety; see \cite{GO10} for details.  
We consider the subvarieties of $\mathsf{RLUG}$ (or $\mathsf{RLUG}^0$) axiomatized by the following equations:
\begin{align*}
(\mathsf{e})\,x \cdot y &\leq y \cdot x & (\mathsf{c})\,x &\leq x \cdot x  & (\mathsf{i})\,x &\leq 1
\end{align*} 
As expected, the equations $(\mathsf{e})$ $(\mathsf{c})$, and $(\mathsf{i})$ correspond to the structural rules $(e)$, $(c)$, and $(i)$, respectively. 
An $r \ell u$-groupoid satisfying the equation $(\mathsf{e})$ (resp. $(\mathsf{c})$, $(\mathsf{i})$) is said to be \emph{commutative} (resp. \emph{square-increasing},  \emph{integral}). 
Given $R \subseteq \{e,c,i\}$, we write $\mathsf{RLUG}_{\mathsf{R}}$ for the subvariety of $\mathsf{RLUG}$ axiomatized by the set $\mathsf{R}$, which consists of the equations corresponding to $R$. 
For instance, if $R=\{e,c\}$, then $\mathsf{RLUG}_{\mathsf{R}}$ forms the variety of square-increasing and commutative $r \ell u$-groupoids. 
An $r \ell uz$-groupoid satisfying the equation $0 \leq x$, denoted by $(\mathsf{o})$, is said to be \emph{zero-bounded}. 
The combination of the equations of $(\mathsf{i})$ and $(\mathsf{o})$ is denoted by $(\mathsf{w})$. 
An \emph{involutive $r\ell uz$-groupoid} is an $r \ell uz$-groupoid $(A,\wedge,\vee,\cdot,\backslash,/,1,0)$ such that for any $x,y,z \in A$, $\negrl x=x=\neglr x$ and $x \cdot y \leq z$ iff $\negl z \cdot x \leq \negl y$ iff $y \cdottt \negr z \leq \negr x$ hold. 
Since the latter condition can be replaced by the equation $\negr x/y=x\backslash\!\!\negl y$, the class $\mathsf{InRLUG}$ of involutive $r \ell uz$-groupoids forms a variety.  
A \emph{cyclic involutive $r \ell uz$-groupoid} is an involutive $r\ell uz$-groupoid satisfying the equation $\negr x=\negl x$.
We write $\mathsf{CyInRLUG}$ for the variety of cyclic involutive $r\ell uz$-groupoids. 

Now we define the validity of $\mathcal{L}^0$-sequents in $r\ell uz$-groupoids. 
Given an $r \ell uz$-groupoid, a map $f \colon V \to A$ is called a \emph{valuation} into $\mathbf{A}$.
This map is uniquely extended to the homomorphism $f \colon \mathbf{Fm}_{\mathcal{L}^0} \to \mathbf{A}$, where $\mathbf{Fm}_{\mathcal{L}^0}$ denotes the absolutely free algebra in $\mathcal{L}^0$ over $V$. 
For a valuation $f$ into an $r\ell uz$-groupoid $\mathbf{A}$, we say that an $\mathcal{L}^0$-sequent $x \Rightarrow a$ is \emph{true} in $\mathbf{A}$ under $f$ and write $\models_{\mathbf{A},f} x \Rightarrow a$, if $f(\rho(x)) \leq f(a)$ holds, where $\rho(x)$ stands for the formula obtained from $x$ by replacing each $\circ$ with $\cdot$. 
(In particular, we put $\rho(x)=1$ if $x=\varepsilon$.) 
Likewise, we say that an $\mathcal{L}^0$-sequent $x \Rightarrow \epsilon$ is true in $\mathbf{A}$ under $f$ and write  $\models_{\mathbf{A},f} x \Rightarrow \epsilon$, if $f(\rho(x)) \leq 0$. 
More generally, given a set $\mathcal{S} \cup \{x \Rightarrow \delta\}$ of $\mathcal{L}^0$-sequents, we write $\mathcal{S} \models_{\mathbf{A},f} x \Rightarrow \delta$, if $\models_{\mathbf{A},f} x \Rightarrow \delta$ holds whenever $\models_{\mathbf{A},f} x' \Rightarrow \delta'$ holds for all $x' \Rightarrow \delta' \in \mathcal{S}$.
We write $\mathcal{S} \models_{\mathbf{A}} x \Rightarrow \delta$ if $\mathcal{S} \models_{\mathbf{A},f} x \Rightarrow \delta$ for all valuation $f$ into $\mathbf{A}$. 
Given a class $\mathcal{K}$ of $r\ell uz$-groupoids, we write $\mathcal{S} \models_{\mathcal{K}} x \Rightarrow \delta$ if $\mathcal{S} \models_{\mathbf{A}} x \Rightarrow \delta$ for any member $\mathbf{A}$ of $\mathcal{K}$.    
Similarly, one defines the validity of $\mathcal{L}$-sequents in $r\ell u$-groupoids.
One proves the following completeness theorems by a standard method. 

\begin{Lem}[\cite{GO10}]
	\label{compmodalfree1}
	Let $\mathcal{S} \cup \{s\}$ be a set of $\mathcal{L}$-sequents and $R \subseteq \{e,c,i\}$. Then,
		\[
		\mathcal{S} \vdash_{\mathbf{FNL}_R} s \iff \mathcal{S} \models_{\mathsf{RLUG}_{\mathsf{R}}}s.
		\]
\end{Lem}

\begin{Lem}[\cite{Bus16,Bus17}]
\label{compmodalfree2}
Let $\mathcal{S} \cup \{s\}$ be a set of $\mathcal{L}^0$-sequents and $R \subseteq \{e,c,w\}$. Then,
	\[
	\mathcal{S} \vdash_{\mathbf{InFNL}_R} s \iff \mathcal{S} \models_{\mathsf{InRLUG}_{\mathsf{R}}}s,
	\]
	and
	\[
	\mathcal{S} \vdash_{\mathbf{CyInFNL}_R} s \iff \mathcal{S} \models_{\mathsf{CyInRLUG}_{\mathsf{R}}}s.
	\]
\end{Lem}

Next, we introduce several types of $\mathcal{L}_{\oc}$-algebras (or $\mathcal{L}^0_{\oc}$-algebras).
An \emph{interior residuated lattice-ordered unital groupoid} (\emph{interior $r\ell u$-groupoid}, for short) is an algebra $(A,\wedge,\vee,\cdot,\backslash,/,\oc,1)$, where $(A,\wedge,\vee,\cdot,\backslash,/,1)$ is an $r\ell u$-groupoid and $\oc$ is a unary operation on $A$ such that for any $x,y \in A$, $1 \leq \oc 1$, $\oc x \cdot \oc y \leq \oc (x \cdot y)$, $\oc x \leq x$, $\oc x \leq \oc \oc x$, and $x \leq y$ implies $\oc x \leq \oc y$. 
In a nutshell, $\oc$ is a \emph{conucleus}. 
The class of interior $r \ell u$-groupoids forms a variety, since the monotonicity of the operation $\oc$ can be replaced with the equation $\oc(x \land y) \leq \oc y$. 
Interior $r \ell uz$-groupoids are defined in an obvious fashion. 
Moreover, we recall some subvarieties of interior $r\ell u(z)$-groupoids, using terminology from \cite{Tan19}. 
An \emph{NACILL-algebra} is an interior $r \ell u$-groupoid satisfying the equations $\oc x \leq 1$, $\oc x \leq \oc x \cdot  \oc x$, $\oc x \cdot y=y \cdot \oc x$, $\oc x \cdot (y \cdot z)=(\oc x \cdot y)\cdot z$, and $x \cdot (y \cdot \oc z)=(x \cdot y)\cdot \oc z$, which are denoted by $(\oc\mathsf{i})$, $(\oc \mathsf{c})$, $(\oc \mathsf{e})$, $(\oc\mathsf{a1})$, and $(\oc\mathsf{a2})$, respectively. 
We write $\mathsf{NACILL}$ for the variety of NACILL-algebras, and $\mathsf{NACILL}_{\mathsf{R}}$ for the subvariety of $\mathsf{NACILL}$ axiomatized by $\mathsf{R} \subseteq \{\mathsf{e},\mathsf{c},\mathsf{i}\}$. 
A member of $\mathsf{NACILL}_{\mathsf{R}}$ is called an \emph{NACILL$_{\mathsf{R}}$}-algebra.
An \emph{NACILL$^0$-algebra} is an algebra $(A,\wedge,\vee,\cdot,\backslash,/,\oc,1,0)$, where $(A,\wedge,\vee,\cdot,\backslash,/,1,0)$ is an $r \ell uz$-groupoid and $\oc$ is a conucleus satisfying the equations $(\oc\mathsf{i})$, $(\oc \mathsf{c})$, $(\oc \mathsf{e})$, $(\oc\mathsf{a1})$, and $(\oc\mathsf{a2})$. 
One defines involutive versions of NACILL$^0$-algebras in a natural way. 
An \emph{NACCLL$^-$-algebra} is an algebra $(A,\wedge,\vee,\cdot,\backslash,/,\oc,1,0)$ such that $(A,\wedge,\vee,\cdot,\backslash,/,1,0)$ is an involutive $r \ell uz$-groupoid and $\oc$ is a conucleus satisfying the equations of $(\oc\mathsf{i})$, $(\oc \mathsf{c})$, $(\oc \mathsf{e})$, $(\oc\mathsf{a1})$, and $(\oc\mathsf{a2})$. 
An \emph{NACCLL}-algebra is just a cyclic NACCLL$^-$-algebra. 
One defines the validity of $\mathcal{L}_{\oc}$-sequents (resp. $\mathcal{L}^0_{\oc}$-sequents) with respect to NACILL-algebras (resp. NACILL$^0$-algebras) in a natural way. 
Similarly to Lemmas~\ref{compmodalfree1} and \ref{compmodalfree2}, one has the following completeness results:
\begin{Lem}
	\label{compNACILL}
	Let $\mathcal{S} \cup \{s\}$ be a set of $\mathcal{L}_{\oc}$-sequents and $R \subseteq \{e,c,i\}$. Then,
				\[
				\mathcal{S} \vdash_{\mathbf{NACILL}_R} s \iff \mathcal{S} \models_{\mathsf{NACILL}_{\mathsf{R}}}s.
				\]
\end{Lem}

\begin{Lem}
	\label{compNACILL0}
	Let $\mathcal{S} \cup \{s\}$ be a set of $\mathcal{L}^0_{\oc}$-sequents and $R \subseteq \{e,c,i,o\}$. Then,
	\[
	\mathcal{S} \vdash_{\mathbf{NACILL}^0_R} s \iff \mathcal{S} \models_{\mathsf{NACILL}^0_{\mathsf{R}}}s.
	\]
\end{Lem}

\begin{Lem}
	\label{compNACCLL-}
	Let $\mathcal{S} \cup \{s\}$ be a set of $\mathcal{L}^0_{\oc}$-sequents and $R \subseteq \{e,c,w\}$. Then,
		\[
		\mathcal{S} \vdash_{\mathbf{NACCLL}^-_R} s \iff \mathcal{S} \models_{\mathsf{NACCLL}^-_{\mathsf{R}}}s,
		\]
		and
		\[
		\mathcal{S} \vdash_{\mathbf{NACCLL}_R} s \iff \mathcal{S} \models_{\mathsf{NACCLL}_{\mathsf{R}}}s.
		\]
\end{Lem}
In the rest of this section, we recall several notions which commonly appear in a wide range of arguments about residuated structures. 
For full discussions about these notions, we refer the reader to \cite{GJ13,GJKO07}.

Let $P$ be a poset. 
A map $cl \colon P \to P$ is called a \emph{closure operator} on $P$ if $x \leq cl(x)$ and $cl(cl(x)) \leq cl(x)$ hold, and $x \leq y$ implies $cl(x) \leq cl(y)$, for any $x,y \in P$. 
If additionally $P$ is a partially-ordered groupoid and $cl(x) \cdot cl(y) \leq cl(x \cdot y)$ holds for all $x,y \in P$, $cl$ is called a \emph{nucleus} on $P$. 
Given an $r\ell u$-groupoid $\mathbf{A}=(A,\land,\lor,\cdot,\backslash,/,1)$ and a nucleus $\gamma$ on $\mathbf{A}$, the algebra $\gamma({\mathbf{A}})=(\gamma[A],\wedge,\vee_{\gamma},\cdot_{\gamma},\backslash,/,\gamma(1))$, where $x \vee_{\gamma} y=\gamma(x \vee y)$ and $x \cdot_{\gamma} y=\gamma(x \cdot y)$, forms an $r \ell u$-groupoid.  

Let $A,B$ be sets, and $R$ a relation between $A$ and $B$. For any $X \in \mathcal{P}(A)$ and $Y \in \mathcal{P}(B)$, we put:
\begin{align*}
X^{\rhd}&:=\{b \in B \mid \forall x \in X (x \R b)\}, \\
Y^{\lhd}&:=\{a \in A \mid \forall y \in Y (a \R y)\}.
\end{align*}
The pair of the maps $^{\rhd} \colon \mathcal{P}(A) \to \mathcal{P}(B)$ and $^{\lhd}\colon \mathcal{P}(B) \to \mathcal{P}(A)$ is known to be a Galois connection. 
Thus the map $\gamma_R$ on $\mathcal{P}(A)$ defined by $\gamma_R(X)=X^{\rhd\lhd}$ forms a closure operator on $\mathcal{P}(A)$. 
$X \in \mathcal{P}(A)$ is called a \emph{closed set} if $X=\gamma_R(X)$. 
A family $\mathcal{X}$ of closed sets is called a \emph{basis} for $\gamma_R$ if any closed set $X$ is equal to intersections of elements of $\mathcal{X}$. 
We end this section by summarizing several basic properties of the maps $^{\rhd}$ and $^{\lhd}$:

\begin{Lem}[\cite{GJ13,GJKO07}]
	\label{closure}
	Let $A,B$ be sets, and $R$ a relation between $A$ and $B$. Then, the following hold:
	\begin{enumerate}
		\item The map $\gamma_R$ is a closure operator on $\mathcal{P}(A)$. 
		\item For any $X_1,X_2 \in \mathcal{P}(A)$, if $X_1 \subseteq X_2$, then $X_2^{\rhd} \subseteq X_1^{\rhd}$. 
		\item For any $Y_1,Y_2 \in \mathcal{P}(B)$, if $Y_1 \subseteq Y_2$, then $Y_2^{\lhd} \subseteq Y_1^{\lhd}$.
		\item For any $X \in \mathcal{P}(A)$ and $Y \in \mathcal{P}(B)$,  $X^{\rhd\lhd\rhd}=X^{\rhd}$ and $Y^{\lhd}=Y^{\lhd\rhd\lhd}$.
		\item For any $X \in \gamma_R[\mathcal{P}(A)]$, $X=\bigcap\{\{b\}^{\lhd} \mid b \in B, X \subseteq \{b\}^{\lhd}\}$, i.e., the set $\{\{b\}^{\lhd} \mid b \in B\}$ is a basis for $\gamma_R$. 
		\item Specifically, if $A$ is a groupoid and $R$ is a nuclear relation, i.e., for any $x,y \in A$ and $z \in B$, there exist elements $x \dbackslash z$ and $z \dslash y$ in $B$ such that
		\[
		x \cdot y \R z \ifff y \R x \dbackslash z \ifff x \R z \dslash y, 
		\]
		then $\gamma_R$ is a nucleus on the powerset groupoid $(\mathcal{P}(A),\circ)$, where $X \circ Y=\{x \cdot y \mid x \in X, y \in Y\}$.
	\end{enumerate}
\end{Lem}  

\section{Decision Problems for Non-associative Intuitionistic Linear Logic and Extensions}
\label{main}
By extending the notion of residuated frame in \cite{GJ13}, we introduce enriched residuated frames.

\begin{Def}
	\normalfont{An \emph{enriched unital residuated frame} (\emph{enriched $ru$-frame}, for short) is a tuple ${\mathbf{F}}=(G,T,N,K)$ such that:
	\begin{itemize}
		\item $(G,T,N)$ is a unital residuated frame, i.e., 
		\begin{itemize}
			\item $G=(G,\cdot,\varepsilon)$ is a unital groupoid, 
			\item $T$ is a set, and 
			\item $N \subseteq G \times T$ is a nuclear relation. 
		\end{itemize}
		\item $K$ is a subunital groupoid of $G$.
	\end{itemize} }
\end{Def}

An \emph{enriched unital residuated frame with zero} (\emph{enriched $ruz$-frame}, for short) is a structure of the form $(G,T,N,K,\epsilon)$ such that $(G,T,N,K)$ an enriched $ru$-frame and $\epsilon \in T$. 
Let ${\mathbf{F}}=(G,T,N,K)$ be an enriched $ru$-frame. 
For any $X,Y \in {\mathcal{P}}(G)$, let:
\begin{align*}
X\backslash Y&:=\{z \mid X \circ \{z\} \subseteq Y\},\\
Y/X&:=\{z \mid \{z\} \circ X \subseteq Y\}, \\
\oc X&:=\gamma_N(X \cap K).
\end{align*}
The definition of the modal operation $\oc$ comes from phase semantics in linear logic; see e.g., \cite{Laf97,OT99,vA05}. 
By Lemma~\ref{closure}, $\gamma_N$ is a nucleus on the powerset $r \ell u$-groupoid $({\mathcal{P}}(G),\cap,\cup,\circ,\backslash,/,\{\varepsilon\})$. 
Hence the dual algebra $\gamma_{N}[\mathcal{P}(\mathbf{G})]=(\gamma_{N}[{\mathcal{P}}(G)],\cap,\cup_{\gamma_{N}},\circ_{\gamma_{N}},\backslash,/,\gamma_{N}(\{\varepsilon\}))$ is a complete $r \ell u$-groupoid. 
Moreover, we define the $\mathcal{L}_{\oc}$-algebra $\mathbf{F}^+=(\gamma_{N}[{\mathcal{P}}(G)],\cap,\cup_{\gamma_{N}},\circ_{\gamma_{N}},\backslash,/,\oc,\gamma_{N}(\{\varepsilon\}))$ by adding the operation $\oc$ to $\gamma_{N}[{\mathcal{P}}(\mathbf{G})]$. 
In what follows, we stipulate that $\mathbf{F}^+$ denotes the $\mathcal{L}^0_{\oc}$-algebra  $(\gamma_{N}[{\mathcal{P}}(G)],\cap,\cup_{\gamma_{N}},\circ_{\gamma_{N}},\backslash,/,\oc,\gamma_{N}(\{\varepsilon\}),\{\epsilon\}^{\lhd})$, whenever the frame $(\mathbf{F},\mathbf{A})$ in question is an enriched $ruz$-frame. 
Then the following holds:
\begin{Thm}[\cite{Tan19}]
	\label{frame}
	If $\mathbf{F}$ is an enriched $ru$-frame (resp. enriched $ruz$-frame), then $\mathbf{F}^+$ is a complete interior $r\ell u$-groupoid (resp. complete interior $r\ell uz$-groupoid).
\end{Thm}

The most natural example of an enriched residuated frame is constructed from an interior $r \ell u$-groupoid. 
Given an interior $r \ell u$-groupoid $\mathbf{A}$, the tuple ${\mathbf{F}_{\mathbf{A}}}=(A,A,\leq^{\mathbf{A}},A^{\oc})$, where $A^{\oc}=\{\oc x \mid x \in A\}$, is an enriched $ru$-frame, since $1 \in A^{\oc}$ and $A^{\oc}$ is closed under multiplication. 
Here, we take $x\dbackslash y=x\backslash^{\mathbf{A}}y$ and $y\dslash x=y/^{\mathbf{A}}x$. 
Then the algebra $\mathbf{F}^+_{\mathbf{A}}=(\gamma_{\leq}[{\mathcal{P}}(A)],\cap,\cup_{\gamma_{\leq}},\circ_{\gamma_{\leq}},\backslash,/,\oc,\gamma_{\leq}(\{1\}))$ is referred to as the \emph{Dedekind-MacNeille completion} of $\mathbf{A}$. 

To give another typical example of an enriched residuated frame, consider an interior $r \ell u$-groupoid $\mathbf{A}$ and a partial subalgebra $\mathbf{B}$ of $\mathbf{A}$. 
We construct the tuple $\mathbf{F}_{\mathbf{A},\mathbf{B}}=(G_B,T_B,N_B,K_B)$ as follows: 
\begin{itemize}
	\item $G_B=(G_B,\cdot,1)$ is the subunital groupoid of $\mathbf{A}$ generated by the set $B$.
	\item $T_B=U_{G_B} \times B$, where $U_{G_B}$ is the set of unary linear polynomials over $G_B$. 
	\item $N_B$ is the binary relation between $G_B$ and $T_B$ such that $x \Nb (u,b)$ iff $u(x) \leq^{\mathbf{A}} b$.
	\item $K_B$ is the subunital groupoid of $\mathbf{A}$ generated by the set $B^{\oc}=\{\oc^{\mathbf{B}}b \mid b \in \dom\oc^{\mathbf{B}}\}$.
\end{itemize}
Given $u \in U_{G_B}$ and $x,y \in G_B$, we write $u(x \cdot\_)$ (resp. $u(\_ \cdot y)$) for the unary linear polynomial which assigns $u(x \cdot y)$ to $y$ (resp. $x$). 
By setting $x \dbackslash (u,b)=(u(x\cdot \_),b)$ and $(u,b) \dslash y=(u(\_\cdot y),b)$, we have $x \cdot y \Nb (u,b)$ iff $y \Nb x \dbackslash (u,b)$ iff $x \Nb (u,b)\dslash y$, for any $x,y \in G_B$ and $(u,b) \in T_B$; thus $N_B$ forms a nuclear relation. 
Therefore $\mathbf{F}_{\mathbf{A},\mathbf{B}}$ is an enriched $ru$-frame. 
\begin{Def}
	\label{enrich}
	\normalfont{An \emph{enriched cut-free Gentzen $ru$-frame} is a tuple $({\mathbf{F}},{\mathbf{A}})$ such that:
	\begin{itemize}
		\item ${\mathbf{F}}=(G,T,N,K)$ is an enriched $ru$-frame.
		\item $\mathbf{A}$ is a partial $\mathcal{L}_{\oc}$-algebra.
		\item There are injections $\zeta \colon A \rightarrow G$, $\xi \colon A \rightarrow T$ and $\kappa \colon A^{\oc} \rightarrow K$, where $A^{\oc}=\{\oc^{\mathbf{A}}a \mid a \in \dom\oc^{\mathbf{A}}\}$. 
		\item For any $x,y \in G$, $z \in T$, $a,a_1,a_2,b \in A$, and $k \in K$, the following rules hold:
		\[
		\begin{bprooftree}
		\AxiomC{}
		\RightLabel{$[1R]$}
		\UnaryInfC{$\varepsilon \N 1^{\mathbf{A}}$}
		\end{bprooftree}
		\begin{bprooftree}
		\AxiomC{}
		\RightLabel{$[Id]$}
		\UnaryInfC{$a \N a$}
		\end{bprooftree}
		\begin{bprooftree}
		\AxiomC{$\varepsilon \N z$}
		\RightLabel{$[1L]$}
		\UnaryInfC{$1^{\mathbf{A}} \N z$}
		\end{bprooftree}
		\]
		\[
		\begin{bprooftree}
		\AxiomC{$a \cdot b \N z$}
		\RightLabel{$[\cdot L]$}
		\UnaryInfC{$a \cdot^{\mathbf{A}} b \N z$}
		\end{bprooftree}
		\begin{bprooftree}
		\AxiomC{$x \N a$}
		\AxiomC{$y \N b$}
		\RightLabel{$[\cdot R]$}
		\BinaryInfC{$x \cdot y \N a \cdot^{\mathbf{A}} b$}
		\end{bprooftree}
		\]
		\[
		\begin{bprooftree}
		\AxiomC{$x \N a$}
		\AxiomC{$b \N z$}
		\RightLabel{$[\backslash L]$}
		\BinaryInfC{$x \cdot (a \backslash^{\mathbf{A}} b) \N z$}
		\end{bprooftree}
		\begin{bprooftree}
		\AxiomC{$a \cdot x \N b$}
		\RightLabel{$[\backslash R]$}
		\UnaryInfC{$x \N a \backslash ^{\mathbf{A}}b$}
		\end{bprooftree}
		\]
		\[
		\begin{bprooftree}
		\AxiomC{$x \N a$}
		\AxiomC{$b \N z$}
		\RightLabel{$[/L]$}
		\BinaryInfC{$(b/^{\mathbf{A}}a)\cdot x \N z$}
		\end{bprooftree}
		\begin{bprooftree}
		\AxiomC{$x \cdot a \N b$}
		\RightLabel{$[/ R]$}
		\UnaryInfC{$x \N b/^{\mathbf{A}}a$}
		\end{bprooftree}
		\]
		\[
		\begin{bprooftree}
		\AxiomC{$a_i \N z$}
		\RightLabel{$[\wedge L]$ for $i=1,2$}
		\UnaryInfC{$a_1 \wedge^{\mathbf{A}} a_2 \N z$}
		\end{bprooftree}
		\begin{bprooftree}
		\AxiomC{$x \N a$}
		\AxiomC{$x \N b$}
		\RightLabel{$[\wedge R]$}
		\BinaryInfC{$x \N a \wedge^{\mathbf{A}} b$}
		\end{bprooftree}
		\]
		\[
		\begin{bprooftree}
		\AxiomC{$a \N z$}
		\AxiomC{$b \N z$}
		\RightLabel{$[\vee L]$}
		\BinaryInfC{$a \vee^{\mathbf{A}} b \N z$}
		\end{bprooftree}
		\begin{bprooftree}
		\AxiomC{$x \N a_i$}
		\RightLabel{$[\vee R]$ for $i=1,2$}
		\UnaryInfC{$x \N a_1 \vee^{\mathbf{A}} a_2$}
		\end{bprooftree}
		\]
		\[
		\begin{bprooftree}
		\AxiomC{$a \N z$}
		\RightLabel{$[\oc L]$}
		\UnaryInfC{$\oc^{\mathbf{A}} a \N z$}
		\end{bprooftree}
		\begin{bprooftree}
		\AxiomC{$k \N a$}
		\RightLabel{$[\oc R]$}
		\UnaryInfC{$k \N \oc^{\mathbf{A}} a$}
		\end{bprooftree}
		\]
	\end{itemize}}
\end{Def}

In the above definition, the third condition says that $A$ is identified with subsets of $G$ and $T$, and that $A^{\oc}$ is identified with a subset of $K$. 
Each of the rules says that, if the upper expression holds then the lower expression also holds. 
For instance, the rule of $[\oc L]$ says that, for any $a \in A$ and $z \in T$, if $\oc^{\mathbf{A}} a$ is defined and $a \N z$ holds, then $\oc^{\mathbf{A}} a \N z$ holds. 
An \emph{enriched cut-free Gentzen $ruz$-frame} $(\mathbf{F},\mathbf{A})$ is a pair $(\mathbf{F},\mathbf{A})$, where $\mathbf{F}$ is an enriched $ruz$-frame, $\mathbf{A}$ is a partial $\mathcal{L}^0_{\oc}$-algebra, and in addition to the rules in Definition~\ref{enrich} the following two rules are also required to hold:
\[
\begin{bprooftree}
\AxiomC{}
\RightLabel{$[0L]$}
\UnaryInfC{$0^{\mathbf{A}} \N \epsilon$}
\end{bprooftree}
\begin{bprooftree}
\AxiomC{$x \N \epsilon$}
\RightLabel{$[0R]$}
\UnaryInfC{$x \N 0^{\mathbf{A}}$}
\end{bprooftree}
\]
An enriched cut-free Gentzen $ru$-frame (resp. enriched cut-free Gentzen $ruz$-frame) $(\mathbf{F},\mathbf{A})$ is called an \emph{enriched Gentzen $ru$-frame} (resp. \emph{enriched Gentzen $ruz$-frame}) if $(\mathbf{F},\mathbf{A})$ satisfies the rule of $[cut]$:
\begin{prooftree}
\AxiomC{$x \N a$}
\AxiomC{$a \N z$}
\RightLabel{$[cut]$}
\BinaryInfC{$x \N z$}
\end{prooftree}
For any $r \ell u$-groupoid (resp. $r \ell uz$-groupoid) $\mathbf{A}$, the pair $(\mathbf{F}_{\mathbf{A}},\mathbf{A})$ forms an enriched Gentzen $ru$-frame (resp. enriched Gentzen $ruz$-frame). 
Also, given an $r \ell u$-groupoid (resp. $r \ell uz$-groupoid) $\mathbf{A}$ and a partial subalgebra $\mathbf{B}$ of $\mathbf{A}$, the pair $(\mathbf{F}_{\mathbf{A},\mathbf{B}},\mathbf{B})$ forms an enriched Gentzen $ru$-frame (resp. enriched Gentzen $ruz$-frame). 
Here, obviously $B \subseteq G_B$, $B^{\oc} \subseteq K_B$. 
Note also that $B$ is identified with the subset $\{(\id,b)\mid b \in B\}$ of $T_B$.
The nuclear relation $N$ of an enriched cut-free Gentzen $ru(z)$-frame $(\mathbf{F},\mathbf{A})$ is said to be \emph{antisymmetric} on $A$ if $a \N b$ and $b \N a$ imply $a=b$ for any $a,b \in A$. 
The following lemma says that, for any enriched cut-free Gentzen $ru(z)$-frame $({\mathbf{F}},{\mathbf{A}})$, $\mathbf{A}$ is \emph{quasi-homomorphic} to $\mathbf{F}^+$. 
\begin{Lem}
	\label{truth}
	Let $({\mathbf{F}},{\mathbf{A}})$ be an enriched cut-free Gentzen $ru$-frame. For every $a,b \in A$ and $X,Y \in \gamma_{N}[{\mathcal{P}}(G)]$, the following hold: 
		\begin{enumerate}
			\item If $1^{\mathbf{A}}$ is defined, $1^{\mathbf{A}} \in \gamma_{N}(\{\varepsilon\}) \subseteq \{1^{\mathbf{A}}\}^{\lhd}$. 
			\item If $a \bullet^{\mathbf{A}} b$ is defined, $a \in X \subseteq \{a\}^{\lhd}$, and $b \in Y \subseteq \{b\}^{\lhd}$, then $a \bullet^{\mathbf{A}} b \in X \bullet^{{\mathbf{F}}^+} Y \subseteq \{a \bullet^{\mathbf{A}} b \}^{\lhd}$, where $\bullet \in \{\wedge,\vee,\cdot,\backslash,/\}$. 
			\item If $\oc^{\mathbf{A}} a$ is defined and $a \in X \subseteq \{a\}^{\lhd}$, then $\oc^{\mathbf{A}} a \in \oc X \subseteq \{\oc^{\mathbf{A}}a\}^{\lhd}$. 
			\item If $(\mathbf{F},\mathbf{A})$ is an enriched cut-free Gentzen $ruz$-frame and $0^{\mathbf{A}}$ is defined, then $0^{\mathbf{A}} \in \{\epsilon\}^{\lhd} \subseteq \{0^{\mathbf{A}}\}^{\lhd}$.
			\item If $({\mathbf{F}},{\mathbf{A}})$ is an enriched Gentzen frame, the map $\{\_\}^{\lhd} \colon A \to \gamma_{N}[\mathcal{P}(G)]$ is a homomorphism from $\mathbf{A}$ to $\mathbf{F}^+$.
			\item If $N$ is antisymmetric on $A$, then the map $\{\_\}^{\lhd} \colon A \to \gamma_{N}[\mathcal{P}(G)]$ is injective.
	\end{enumerate}
\end{Lem}
\begin{proof}
	For the proofs of all the statements except for Statement (3), see \cite[Theorem 2.5]{GJ13}. 
	We show Statement (3). 
	Let $z \in X^{\rhd}$.
	We have $a \N z$ by the assumption. 
	Since $\oc^{\mathbf{A}} a$ is defined, we have $\oc^{\mathbf{A}} a \N z$, using the rule of $[\oc L]$.
	Hence $\oc^{\mathbf{A}} a \in X$.
	On the other hand, $\oc^{\mathbf{A}} a \in K$. 
	So, $\oc^{\mathbf{A}} a \in X \cap K \subseteq \oc X$.
	Let $k \in X \cap K$.
	Since $k \in X \cap K \subseteq X \subseteq \{a\}^{\lhd}$, we have $k \N a$.
	Since $\oc^{\mathbf{A}} a$ is defined, by using the rule of $[\oc R]$, we have $k \N \oc^{\mathbf{A}} a$; thus $k \in \{\oc^{\mathbf{A}} a\}^{\lhd}$. 
	Therefore, we have $X \cap K \subseteq \{\oc^{\mathbf{A}} a\}^{\lhd}$. 
    Thus $\oc X \subseteq \{\oc^{\mathbf{A}} a\}^{\lhd}$ follows by properties of nuclei.
\end{proof}

Immediately, we have the following corollaries: 
\begin{Cor}
	\label{embedd}
	Let $(\mathbf{F},\mathbf{A})$ be an enriched Gentzen $ru$-frame (or enriched Gentzen $ruz$-frame) whose nuclear relation is antisymmetric on $A$. The map $\{\_\}^{\lhd} \colon A \to \gamma_N[\mathcal{P}(G)]$ is an embedding of $\mathbf{A}$ into $\mathbf{F}^+$.
\end{Cor}

\begin{Cor}
	\label{embed1}
	Let $\mathbf{A}$ be an interior $r \ell u$-groupoid (or interior $r \ell uz$-groupoid).  The map $\{\_\}^{\lhd} \colon A \to \gamma_{\leq}[\mathcal{P}(A)]$ is an embedding of $\mathbf{A}$ into $\mathbf{F}_{\mathbf{A}}^+$.
\end{Cor}

\begin{Cor}
	\label{corembed2}
	Let $\mathbf{A}$ be an interior $r \ell u$-groupoid (or interior $r \ell uz$-groupoid) and $\mathbf{B}$ a partial subalgebra of $\mathbf{A}$. The map $\{\_\}^{\lhd} \colon B \to \gamma_{N_B}[\mathcal{P}(G_B)]$ is an embedding of $\mathbf{B}$ into $\mathbf{F}_{\mathbf{A},\mathbf{B}}^+$.
\end{Cor}

Moreover, one can optionally add some of the following extra rules to enriched cut-free Gentzen frames:
\[
\begin{bprooftree}
\AxiomC{$\varepsilon \N z$}
\RightLabel{$[ki]$}
\UnaryInfC{$k \N z$}
\end{bprooftree}
\begin{bprooftree}
\AxiomC{$k \cdot k \N z$}
\RightLabel{$[kc]$}
\UnaryInfC{$k \N z$}
\end{bprooftree}
\]
\[
\begin{bprooftree}
\AxiomC{$k \cdot y \N z$}
\doubleLine
\RightLabel{$[ke]$}
\UnaryInfC{$y \cdot k \N z$}
\end{bprooftree}
\begin{bprooftree}
\AxiomC{$k \cdot (x \cdot y) \N z$}
\doubleLine
\RightLabel{$[ka1]$}
\UnaryInfC{$(k \cdot x) \cdot y \N z$}
\end{bprooftree}
\begin{bprooftree}
\AxiomC{$x \cdot (y \cdot k) \N z$}
\doubleLine
\RightLabel{$[ka2]$}
\UnaryInfC{$(x \cdot y) \cdot k \N z$}
\end{bprooftree}
\]
\[
\begin{bprooftree}
\AxiomC{$x \cdot y \N z$}
\RightLabel{$[e]$}
\UnaryInfC{$y \cdot x \N z$}
\end{bprooftree}
\begin{bprooftree}
\AxiomC{$x \cdot x \N z$}
\RightLabel{$[c]$}
\UnaryInfC{$x \N z$}
\end{bprooftree}
\begin{bprooftree}
\AxiomC{$\varepsilon \N z$}
\RightLabel{$[i]$}
\UnaryInfC{$x \N z$}
\end{bprooftree}
\]
As shown in the lemma below, these rules are closely related to the equations introduced in Section~\ref{pre}.
\begin{Lem}
	\label{rule}
	Let $(\mathbf{F},\mathbf{A})$ be an enriched cut-free Gentzen $ru(z)$-frame.
		$(\mathbf{F},\mathbf{A})$ satisfies the rule $[e]$ (resp. $[c]$, $[i]$, $[kc]$, $[ke]$, $[ki]$, $[ka1]$, $[ka2]$) if and only if the equation $(\mathsf{e})$ (resp. $(\mathsf{c})$, $(\mathsf{i})$, $(\oc\mathsf{c})$, $(\oc\mathsf{e})$, $(\oc \mathsf{i})$, $(\oc\mathsf{a1})$, $(\oc\mathsf{a2})$) holds in $\mathbf{F}^+$.	
\end{Lem}

\begin{proof}
	We show that, the rule of [$ka1$] holds in $(\mathbf{F},\mathbf{A})$ if and only if the equation $(\oc\mathsf{a1})$ holds in $\mathbf{F}^{+}$, i.e., $\oc X \circ_{\gamma_N} (Y \circ_{\gamma_N} Z) = (\oc X \circ_{\gamma_N} Y) \circ_{\gamma_N} Z$ holds for every $X,Y,Z \in \gamma_{N}[\mathcal{P}(G)]$.  
	For the \emph{only-if} direction, let $X,Y,Z \in \gamma_{N}[\mathcal{P}(G)]$.
	Suppose that $k \cdot (y \cdot z) \in X \cap K \circ (Y \circ Z)$. 
	Then it is clear that $(k \cdot y) \cdot z \in (\oc X \circ_{\gamma_N} Y) \circ_{\gamma_N} Z$. 
	Let $t \in [(\oc X \circ_{\gamma_N} Y) \circ Z]^{\rhd}$; hence $(k \cdot y) \cdot z \N t$. 
	By using [$ka1$], we have $k \cdot (y \cdot z) \N t$; thus $k \cdot (y \cdot z) \in [(\oc X \circ_{\gamma_N} Y) \circ Z]^{\rhd\lhd}=(\oc X \circ_{\gamma_N} Y) \circ_{\gamma_N} Z$. 
	So, $ X \cap K \circ (Y \circ Z) \subseteq (\oc X \circ_{\gamma_N} Y) \circ_{\gamma_N} Z$. 
	Using the properties of nuclei, we have $\oc X \circ_{\gamma_N} (Y \circ_{\gamma_N} Z) \subseteq (\oc X \circ_{\gamma_N} Y) \circ_{\gamma_N} Z$. 
	One checks that the reverse inclusion also holds in $\mathbf{F}^+$ in a similar way. 
	
	For the \emph{if} direction, suppose that $k \cdot (x \cdot y) \N z$; thus $\{k\} \circ (\{x\} \circ \{y\}) \subseteq \{z\}^{\lhd}$. 
	By using the properties of nuclei and Lemma~\ref{closure}, $\gamma_{N}(\{k\}) \circ_{\gamma_N} [\gamma_{N}(\{x\}) \circ_{\gamma_N} \gamma_N(\{y\})] \subseteq \{z\}^{\lhd}$. 
	Due to the fact that $\oc \gamma_N(\{k\})=\gamma_N(\{k\})$, we have: 
	\begin{align*}
	(k \cdot x) \cdot y &\in [\gamma_{N}(\{k\}) \circ_{\gamma_N} \gamma_{N}(\{x\})] \circ_{\gamma_N} \gamma_N(\{y\})\\
	&=[\oc \gamma_N(\{k\}) \circ_{\gamma_N} \gamma_{N}(\{x\})] \circ_{\gamma_N} \gamma_N(\{y\}) \\
	&= \oc \gamma_N(\{k\}) \circ_{\gamma_N} [\gamma_{N}(\{x\}) \circ_{\gamma_N} \gamma_N(\{y\})] \tag{by $(\oc \mathsf{a1})$}\\
	&= \gamma_N(\{k\}) \circ_{\gamma_N} [\gamma_{N}(\{x\}) \circ_{\gamma_N} \gamma_N(\{y\})]\\
	&\subseteq \{z\}^{\lhd}
	\end{align*}
	Hence we have $(k \cdot x) \cdot y \N z$. 
	In a similar way, we prove that $(k \cdot x) \cdot y \N z$ implies $k \cdot (x \cdot y) \N z$. 
	The verification of the remaining cases is left to the reader.
\end{proof}
Clearly, if the equation $(\mathsf{e})$ (resp. $(\mathsf{c})$, $(\mathsf{i})$, $(\oc\mathsf{c})$, $(\oc\mathsf{e})$, $(\oc \mathsf{i})$, $(\oc\mathsf{a1})$, $(\oc\mathsf{a2})$) holds in an interior $r \ell u(z)$-groupoid $\mathbf{A}$, then the enriched Gentzen $ru(z)$-frame $(\mathbf{F}_{\mathbf{A}},\mathbf{A})$ satisfies the rule $[e]$ (resp. $[c]$, $[i]$, $[kc]$, $[ke]$, $[ki]$, $[ka1]$, $[ka2]$). 
Likewise, for any interior $r \ell u(z)$-groupoid $\mathbf{A}$ and a partial subalgebra $\mathbf{B}$ of $\mathbf{A}$, it can be easily checked that, if the equation $(\mathsf{e})$ (resp. $(\mathsf{c})$, $(\mathsf{i})$, $(\oc\mathsf{c})$, $(\oc\mathsf{e})$, $(\oc \mathsf{i})$ $(\oc\mathsf{a1})$, $(\oc\mathsf{a2})$) holds in $\mathbf{A}$, then the enriched Gentzen $ru(z)$-frame $(\mathbf{F}_{\mathbf{A},\mathbf{B}},\mathbf{B})$ satisfies the rule $[e]$ (resp. $[c]$, $[i]$, $[kc]$, $[ke]$, $[ki]$, $[ka1]$, $[ka2]$). 
By Lemma~\ref{rule}, we have the two corollaries below:    

\begin{Cor}
	\label{preserve1}
	Let $\mathbf{A}$ be an interior $r\ell u(z)$-groupoid and $(eq)$ any of the equations $(\mathsf{e})$, $(\mathsf{c})$, $(\mathsf{i})$, $(\oc\mathsf{c})$, $(\oc\mathsf{e})$, $(\oc \mathsf{i})$, $(\oc\mathsf{a1})$, and $(\oc\mathsf{a2})$. 
		If $(eq)$ holds in $\mathbf{A}$, then it also holds in $\mathbf{F}_{\mathbf{A}}^+$.
\end{Cor}
\begin{Cor}
	\label{preserve2}
	Let $\mathbf{A}$ be an interior $r\ell u(z)$-groupoid, $\mathbf{B}$ a partial subalgebra of $\mathbf{A}$, and $(eq)$ any of the equations $(\mathsf{e})$, $(\mathsf{c})$, $(\mathsf{i})$, $(\oc\mathsf{c})$, $(\oc\mathsf{e})$, $(\oc \mathsf{i})$, $(\oc\mathsf{a1})$, and $(\oc\mathsf{a2})$. If $(eq)$ holds in $\mathbf{A}$, then it also holds in $\mathbf{F}_{\mathbf{A},\mathbf{B}}^+$.
\end{Cor}

A class $\mathcal{K}$ of algebras has the \emph{finite embeddability property} (\emph{FEP}, for short) if every finite partial subalgebra of a member of $\mathcal{K}$ is embedded into a finite member of $\mathcal{K}$. 
If a class of algebras is finitely axiomatizable and has the FEP, then it has the decidable universal theory; see \cite{BvA02,GJKO07} for details. 
We say that a class $\mathcal{K}$ of algebras has the \emph{strong finite model property} (\emph{SFMP}, for short) if every quasiequation refuted in $\mathcal{K}$ is refuted in a finite member of $\mathcal{K}$. 
If $\mathcal{K}$ is finitely axiomatizable and has the SFMP, then its quasiequational theory is decidable.  
As for quasivarieties of finite type, the FEP is known to be equivalent to the SFMP; see \cite[Lemma 6.40]{GJKO07}. 
Now we are ready to show the FEP for integral interior $r\ell u$-groupoids.
\begin{Lem} 
\label{finite1}
Let $\mathbf{A}$ be an integral interior $r\ell u$-groupoid and $\mathbf{B}$ a finite partial subalgebra of $\mathbf{A}$. Then $\mathbf{F}^+_{\mathbf{A},\mathbf{B}}$ is finite.
\end{Lem}

\begin{proof}
	The techniques by Blok-van Alten \cite{BvA04} and Galatos-Jipsen \cite{GJ13} also work well in our setting. 
	The proof here does not require a discussion about the modal operation, and thus is essentially the same as the proofs of \cite[Lemma 3.6]{BvA04} and \cite[Theorem 3.18]{GJ13}.

	Let $\mathbf{A}$ be an integral interior $r \ell u$-groupoid and $\mathbf{B}$ a finite partial subalgebra $\mathbf{B}$ of $\mathbf{A}$, where $B=\{b_1,\ldots,b_n\}$. 
	For our purpose here, in view of Lemma~\ref{closure}, it suffices to show that the collection $\{\{z\}^{\lhd} \mid z \in T_B\}$ is finite.
	Consider the free unital groupoid $(G_X,\cdot^{G_X},1^{G_X})$ generated by the set $X=\{x_1,\ldots,x_n\}$. 
	For any $s,t \in G_X$, we set $s \leq^{G_X} t$ iff $t$ is obtained from $s$ by deleting some (possibly none) of generators. 
	Specifically, $s \leq^{G_X} 1^{G_X}$ for all $s \in G_X$. 
	For any $s,t \in G_X$, the set $\{x \in G_X \mid s \cdot^{G_X} x \leq^{G_X} t\}$ (resp. $\{x \in G_X \mid x \cdot^{G_X} s \leq^{G_X} t\}$) possesses the greatest element, denoted by $s \backslash^{G_X} t$ (resp. $t/^{G_X}s$). 
	Actually, the structure $(G_X,\leq^{G_X},\cdot^{G_X},\backslash^{G_X},/^{G_X},1^{G_X})$ forms an integral residuated partially ordered unital groupoid, and is  well-quasi-ordered; refer to \cite[Section 3]{BvA04} for details. 
	Define the map $h \colon X \to B$ by $x_i \mapsto b_i$ for each $i \in \{1,\ldots,n\}$.
	Thanks to the integrality of the partially-ordered unital groupoid $G_B$ whose order is induced by $\leq^{\mathbf{A}}$, this map is extended to the order-preserving surjective homomorphism from $G_X$ to $G_B$.  
	Furthermore, this map is extended to the surjection from $U_{G_X}$ to $U_{G_B}$. 
	We define the binary relation $R_h \subseteq G_X \times T_B$ by setting $x \Nh z$ iff $h(x) \Nb z$. 
	By setting $x \dbackslash_h z= h(x) \dbackslash z$ and $z \dslash_h y=z\dslash h(y)$, 
	one has $x \cdot y \Nh z$ iff $y \Nh x\dbackslash_h z$ iff $x \Nh z\dslash_h y$; hence the tuple $(G_X,T_B,N_h)$ becomes a unital residuated frame. 
	Then one has $s \in \{(u,b)\}^{\lhd_{N_h}}$ iff $v(s) \in h^{-1}(\darrowa b)$, for every $s \in G_X$, every $(u,b) \in T_B$, and some $v \in U_{G_X}$ such that $h(v)=u$. 
	Here, $\darrowa b$ denotes the downward closure of $b$ in $\mathbf{A}$.
	In fact, $s \in \{(u,b)\}^{\lhd_{N_h}}$ iff $s \Nh (u,b)$ iff $h(s) \Nb (u,b)$ iff $u(h(s)) \leq^{\mathbf{A}} b$ iff $h(v)(h(s)) \leq^{\mathbf{A}} b$ iff $h(v(s)) \leq^{\mathbf{A}} b$ iff $h(v(s)) \in \darrowa b$ iff $v(s) \in h^{-1}(\darrowa b)$. 
	Then $h^{-1}(\darrowa b)=\darrowgx \Crit(b)$, where $\Crit(b)$ denotes the (finite) set of maximal elements of $h^{-1}(\darrowa b)$, since $h^{-1}(\darrowa b)$ is a downward closed subset of $G_X$ and $G_X$ is well-quasi-ordered; see \cite[Section~3]{BvA04} for details. 
	Hence $s \in \{(u,b)\}^{\lhd_{N_h}}$ iff $v(s) \in \darrowgx \Crit(b)$ iff $v(s) \leq^{G_X} t$ for some $t \in \Crit(b)$. 
	For any $u \in U_{G_X}$ and $s \in G_X$, we define the element $\sigma^u(s)$ in $G_X$ by induction on $u$ as follows:
	 \[
	 \sigma^u(s):=
	 \begin{cases}
	 	s & \text{if $u=\id$}\\
	 	\sigma^v(s/^{G_X}t) & \text{if $u=v \cdot t$} \\	
	 	\sigma^v(t\backslash^{G_X} s) & \text{if $u=t \cdot v$}
	 \end{cases}
	 \]
	 By induction on $u$, one proves that $u(s)\leq^{G_X} t$ iff $s \leq^{G_X} \sigma^u(t)$ holds for any $s,t \in G_X$ and $u \in U_{G_X}$. 
	 Therefore we have $s \in \{(u,b)\}^{\lhd_{N_h}}$ iff $s \leq^{G_X} \sigma^v(t)$, i.e., $\{(u,b)\}^{\lhd_{N_h}}=\darrowgx C$, where $C=\{\sigma^v(t) \mid t \in \Crit(b)\}$. 
	 Since we have $t \leq^{G_X} \sigma^v(t)$ for any $t \in \Crit(b)$, using the integrality of $G_X$, we have $C \subseteq \bigcup_{b \in B} \bigcup_{t \in \Crit(b)} \uarrowgx t$. 
	 Here, the set $\bigcup_{b \in B} \bigcup_{t \in \Crit(b)} \uarrowgx t$ is finite, since the sets $\uarrowgx t$, $B$, and $\Crit(b)$ are all finite; thus so is $C$.
	 Observe that $h[\{(u,b)\}^{\lhd_{N_h}}]=\{(u,b)\}^{\lhd}$ for any $(u,b) \in T_B$. 
	 It follows that, for any $(u,b) \in T_B$, there exists a finite set $D$ such that $D \subseteq \bigcup_{b \in B} \bigcup_{t \in \Crit(b)} \uarrowgx t$ and $h(\darrowgx D)=\{(u,b)\}^{\lhd}$, i.e., there are at most finitely many sets of the form $\{(u,b)\}^{\lhd}$. 
\end{proof}

By Corollaries~\ref{corembed2}, \ref{preserve2} and Lemma~\ref{finite1}, we have:
\begin{Cor}
	\label{FEPnonzero}
	Let $\mathcal{V}$ be a subvariety of integral interior $r \ell u$-groupoids axiomatized by any combination of the equations $(\mathsf{e})$, $(\mathsf{c})$,  $(\oc\mathsf{e})$, $(\oc\mathsf{c})$, $(\oc\mathsf{a1})$, and $(\oc\mathsf{a2})$. $\mathcal{V}$ has the FEP.
\end{Cor}
Specifically, it follows that the varieties $\mathsf{NACILL}_{\mathsf{i}}$, $\mathsf{NACILL}_{\mathsf{ei}}$, and $\mathsf{NACILL}_{\mathsf{ci}}$ have the FEP, whereas $\mathsf{NACILL}$, $\mathsf{NACILL}_{\mathsf{e}}$, $\mathsf{NACILL}_{\mathsf{c}}$, and $\mathsf{NACILL}_{\mathsf{ec}}$ have the undecidable equational theories (cf. Theorem~\ref{ec}).  
This situation is analogous to the two facts that (i) intuitionistic linear logic is undecidable and (ii) the variety corresponding to intuitionistic linear logic with left-weakening (i.e. intuitionistic affine logic) has the FEP; see \cite{LMSS92,vA05}.
Notice that $\mathsf{NACILL}_{\mathsf{eci}}$ boils down to $\mathsf{NACILL}_{\mathsf{ci}}$, since the combination of the properties of integrality and being square-increasing restores commutativity (and associativity). 
On the other hand, we have:
\begin{Thm}[\cite{Tan19}]
	\label{ec}
	Let $R \subseteq \{e,c\}$. The provability problem for $\mathbf{NACILL}_R$ is undecidable.
\end{Thm}

Using this theorem, we conclude:
\begin{Thm}
	\label{central1}
	Let $R \subseteq \{e,c,i\}$. The following statements are mutually equivalent:
		\begin{enumerate}
			\item $i \in R$,
			\item the variety $\mathsf{NACILL}_\mathsf{R}$ has the FEP, 
			\item the variety $\mathsf{NACILL}_\mathsf{R}$ has the SFMP,
			\item the deducibility problem for $\mathbf{NACILL}_R$ is decidable, 
			\item the provability problem for $\mathbf{NACILL}_R$ is decidable. 
	\end{enumerate}
\end{Thm}
\begin{proof}
	\begin{description}
		\item[\normalfont$(1)\Rightarrow(2)$:] By Corollary~\ref{FEPnonzero}.
		\item[\normalfont$(2)\Rightarrow(3)$:] By \cite[Lemma~6.40]{GJKO07}.
		\item[\normalfont$(3)\Rightarrow(4)$:] By Lemma~\ref{compNACILL}.
		\item[\normalfont$(4)\Rightarrow(5)$:] Trivial.
		\item[\normalfont$(5)\Rightarrow(1)$:] By Theorem~\ref{ec}.
	\end{description}
\end{proof}

Here, it should be noted that the above equivalence collapses when associativity is also concerned; in fact, the variety $\mathsf{NACILL}_{\mathsf{aec}}$ has the FEP (cf. \cite{vA05}). 
The remaining part of this section is devoted to the proof of the FEP for integral  interior $r \ell uz$-groupoids. 
Given an interior $r \ell uz$-groupoid $\mathbf{A}$ and a finite partial subalgebra $\mathbf{B}$ of $\mathbf{A}$, we write $\mathbf{B}^0$ for the finite partial subalgebra of $\mathbf{A}$ whose underlying set is $B^0=B\cup \{0\}$. 
Of course, the tuple $(\mathbf{F}_{\mathbf{A},\mathbf{B}^0},\mathbf{B}^0)$ forms an enriched Gentzen $ruz$-frame, where  $\mathbf{F}_{\mathbf{A},\mathbf{B}^0}=(G_{B^0},T_{B^0},N_{B^0},K_{B^0},\epsilon)$ is an enriched $ruz$-frame such that $\epsilon=(\id,0)$.

\begin{Lem}
	\label{zerobounded}
	Let $\mathbf{A}$ be an integral interior $r\ell uz$-groupoid and $\mathbf{B}$ a finite partial subalgebra of $\mathbf{A}$. 
	Then $\mathbf{F}^+_{\mathbf{A},\mathbf{B}^0}$ is finite. 
	In addition, if $\mathbf{A}$ is zero-bounded, then so is $\mathbf{F}^+_{\mathbf{A},\mathbf{B}^0}$.
\end{Lem}
\begin{proof}
	The proof is a slightly modified version of that of \cite[Theorem 6.46]{GJKO07}.
	By repeating the argument in the proof of Lemma~\ref{finite1}, one confirms the finiteness of $\mathbf{F}^+_{\mathbf{A},\mathbf{B}^0}$.
	For the remaining claim, suppose that $0$ is the least element of $\mathbf{A}$. 
	Obviously, $u(0) \leq^{\mathbf{A}}b$ holds for any $u \in U_{G_{B^0}}$ and $b \in B^0$. 
	Thus $0 \Nz (u,b)$ for all $(u,b) \in T_{B^0}$.
	So, $0 \in \{(u,b)\}^{\lhd}$ for all $(u,b) \in T_{B^0}$, i.e., $0 \in \bigcap \mathcal{X}$, where $\mathcal{X}$ denotes the basis $\{\{(u,b)\}^{\lhd} \mid (u,b) \in T_{B^0}\}$ for $\gamma_{N_{B^0}}$. 
	Let $x \in \{\epsilon\}^{\lhd}$, i.e., $x \Nz (\id,0)$.
	We have $\id(x) \leq^{\mathbf{A}}0$, i.e., $x \leq^{\mathbf{A}} 0$. 
	Since $\mathbf{A}$ is zero-bounded, we have $x=0$, i.e., $\{\epsilon\}^{\lhd} \subseteq \{0\} \subseteq \bigcap \mathcal{X}$. 
	It turns out that $\{\epsilon\}^{\lhd}$ is the smallest closed set in $\mathbf{F}^+_{\mathbf{A},\mathbf{B}^0}$. 
\end{proof}
Recall that by Corollary~\ref{preserve2}, if the equations $(\mathsf{e})$, $(\mathsf{c})$, $(\mathsf{i})$,  $(\oc\mathsf{e})$, $(\oc\mathsf{c})$, $(\oc \mathsf{i})$, $(\oc\mathsf{a1})$, and $(\oc\mathsf{a2})$ holds in $\mathbf{A}$, then they hold in  $\mathbf{F}_{\mathbf{A},\mathbf{B}^0}^+$.
By Lemma~\ref{corembed2}, $\mathbf{B}^0$ is embeddable into $\mathbf{F}^+_{\mathbf{A},\mathbf{B}^0}$; thus $\mathbf{B}$ is of course embeddable into $\mathbf{F}^+_{\mathbf{A},\mathbf{B}^0}$. 
Hence we have:
\begin{Cor}
\label{FEP}
Let $\mathcal{V}$ be a subvariety of integral interior $r \ell uz$-groupoid axiomatized by any combination of $(\mathsf{e})$, $(\mathsf{c})$, $(\mathsf{o})$,  $(\oc\mathsf{e})$, $(\oc\mathsf{c})$, $(\oc\mathsf{a1})$, and $(\oc\mathsf{a2})$. $\mathcal{V}$ has the FEP.
\end{Cor}

For instance, the varieties $\mathsf{NACILL}^0_{\mathsf{i}}$, $\mathsf{NACILL}^0_{\mathsf{ei}}$, $\mathsf{NACILL}^0_{\mathsf{ci}}$, $\mathsf{NACILL}^0_{\mathsf{io}}$, $\mathsf{NACILL}^0_{\mathsf{eio}}$, and $\mathsf{NACILL}^0_{\mathsf{cio}}$ have the FEP.  
Here, $\mathsf{NACILL}^0_{\mathsf{cio}}$ is nothing but the algebraic models for the intuitionistic version of modal logic $\mathbf{S4}$.
Using Theorem~\ref{ec}, we prove: 
\begin{Thm}
	\label{und}
	Let $R \subseteq \{e,c,o\}$. The provability problem for $\mathbf{NACILL}^0_R$ is undecidable. 
\end{Thm}
\begin{proof}
	If $o \not \in R$, by Theorem~\ref{ec}, the provability problem for $\mathbf{NACILL}^0_R$ is clearly undecidable, because every NACILL$_\mathsf{R}$-algebra is trivially the 0-free reduct of an NACILL$^0_\mathsf{R}$-algebra.
	For the case where $o \in R$, by Theorem~\ref{ec}, it suffices to confirm that every NACILL$_{\mathsf{R}^-}$-algebra is embedded into the 0-free reduct of an NACILL$^0_{\mathsf{R}}$-algebra, where $\mathsf{R}^-=\mathsf{R}-\{\mathsf{o}\}$.
	Let $\mathbf{A}$ be a member of $\mathsf{NACILL}_{\mathsf{R}^-}$. 
	By Corollaries~\ref{preserve1} and \ref{embed1}, $\mathbf{A}$ is embedded into $\mathbf{F}_{\mathbf{A}}^+$ and $\mathbf{F}_{\mathbf{A}}^+$ is also a member of $\mathsf{NACILL}_{\mathsf{R}^-}$. 
	Clearly, the algebra $\mathbf{F}_{\mathbf{A}}^{+0}=(\gamma_{\leq}[{\mathcal{P}}(G)],\cap,\cup_{\gamma_{\leq}},\circ_{\gamma_{\leq}},\backslash,/,\oc,\gamma_{\leq}(\{1\}),\gamma_{\leq}(\varnothing))$ is a member of $\mathsf{NACILL}^0_{\mathsf{R}}$, and $\mathbf{A}$ is embeddable into the 0-free reduct of $\mathbf{F}_{\mathbf{A}}^{+0}$.
\end{proof}

\begin{Thm}
	\label{central}
	Let $R \subseteq \{e,c,i,o\}$. The following statements are mutually equivalent:
	\begin{enumerate}
		\item $i \in R$,
		\item the variety $\mathsf{NACILL}^0_\mathsf{R}$ has the FEP, 
		\item the variety $\mathsf{NACILL}^0_\mathsf{R}$ has the SFMP,
		\item the deducibility problem for $\mathbf{NACILL}^0_R$ is decidable, 
		\item the provability problem for $\mathbf{NACILL}^0_R$ is decidable. 
	\end{enumerate}
\end{Thm}
\begin{proof}
	\begin{description}
		\item[\normalfont$(1)\Rightarrow(2)$:] By Corollary~\ref{FEP}.
		\item[\normalfont$(2)\Rightarrow(3)$:] By \cite[Lemma~6.40]{GJKO07}.
		\item[\normalfont$(3)\Rightarrow(4)$:] By Lemma~\ref{compNACILL0}.
		\item[\normalfont$(4)\Rightarrow(5)$:] Trivial.
		\item[\normalfont$(5)\Rightarrow(1)$:] By Theorem~\ref{und}.
	\end{description}
\end{proof}

\section{Undecidability of Non-associative Classical Linear Logic and Extensions}
\label{classical}

Roughly speaking, this section is divided into two parts. 
Firstly, we show that the extensions of $\mathbf{InFNL}$ and $\mathbf{CyInFNL}$ by the rules of contraction and exchange are all undecidable (Theorem~\ref{FCNL}). 
Secondly, using the undecidability result from the first part, we show that the provability problems for all the extensions of $\mathbf{NACCLL}^-$ and $\mathbf{NACCLL}$ by contraction and exchange are undecidable (Theorem~\ref{classicalundecidable}). 
For our purposes, the following theorem plays a crucial role.  

\begin{Thm}
	\label{inv}
	Let $\mathsf{R} \subseteq \{\mathsf{e},\mathsf{c}\}$. Every member of $\mathsf{RLUG}_{\mathsf{R}}$ is a subalgebra of the $\mathcal{L}$-reduct of a member of $\mathsf{CyInRLUG}_{\mathsf{R}}$.
\end{Thm}
\begin{proof}	
	Let $\mathbf{A}=(A,\land,\lor,\cdot,\backslash,/,1)$ be a member of $\mathsf{RLUG}_{\mathsf{R}}$.
	We use the construction given in the proof of \cite[Theorem 9.1]{GR04}. 
	First of all, we define the groupoid $(A^{\star},\cdot^{\star})$ as follows:
	\begin{itemize}
		\item $A^{\star}:=A \cup A^{\sim} \cup \{\top,\bot\}$, where $A^{\sim}=\{x^{\sim} \mid x \in A\}$ is a bijective copy of $A$ such that $A \cap A^{\sim}=\varnothing$, and $\bot,\top \not \in A \cup A^{\sim}$. 
		\item The multiplication $\cdot^{\star}$ on $A^{\star}$ is defined as follows: for any $x,y \in A$ and $z \in A^{\star}$,  
		\begin{align*}
		x \cdot^{\star} y&=x \cdot y & x \cdot^{\star} y^{\sim}&=(y/x)^{\sim}\\
		x^{\sim} \cdot^{\star} y&= (y \backslash x)^{\sim} & x^{\sim} \cdot^{\star} y^{\sim}&=\top \\
		z \cdot^{\star} \top&=\top \cdot^{\star} z=
		\begin{cases}
		\top & \text{if $z\not=\bot$}\\
		\bot & \text{if $z=\bot$}
		\end{cases}
		& z\cdot^{\star} \bot&=\bot \cdot^{\star} z=\bot
		\end{align*}
		\end{itemize}
	    The lattice-order $\leq$ on $A$ induced by $\land$ and $\lor$ is extended to the partial order $\leq^{\star}$ on $A^{\star}$ as follows: for any $x,y \in A$,
	    \begin{align*}
	    & \bot <^{\star} x <^{\star} y^{\sim} <^{\star} \top, & x^{\sim} \leq^{\star} y^{\sim} \ifff y \leq x. 
	    \end{align*}
		Define the unary operation $\sim$ on $A^{\star}$ by:
		\[
		\negr x=
		\begin{cases}
			x^{\sim} & \text{if $x \in A$}\\
			y & \text{if $x=y^{\sim}$ for some $y \in A$}\\
			\bot & \text{if $x=\top$}\\
			\top & \text{if $x=\bot$}
		\end{cases}
		\]
		Also, we set $1^{\star}=1$. 
		Clearly, $\negrr x=x$ for all $x \in A^{\star}$.  
		The following holds:
		\begin{Lem}
			\label{appa}
			\begin{enumerate}
				\item $(A^{\star},\leq^{\star},\cdot^{\star},1^{\star})$ is a lattice-ordered unital groupoid.
				\item $x \cdot^{\star} y \leq^{\star} z$ iff $\negr z \cdot^{\star} x \leq^{\star} \negr y$ iff $y \cdott \negr z \leq^{\star} \negr x$, for any $x,y,z \in A^{\star}$.
			\end{enumerate}
		\end{Lem}

	    For simplicity of the argument here, we prove this lemma in \ref{appendixa}. 
	    Using Lemma~\ref{appa}, we have $x \leq^{\star} y$ iff $\negr y \leq^{\star} \negr x$ for any $x,y \in A^{\star}$, since $x \leq^{\star} y$ iff $1^{\star} \cdot^{\star} x \leq^{\star} y$ iff $\negr y \cdot^{\star} 1^{\star} \leq^{\star} \negr x$ iff $\negr y \leq^{\star} \negr x$ for any $x,y \in A^{\star}$.
	    We define the binary operations $\backslash^{\star}$ and $/^{\star}$ on $A^{\star}$ by: 
	    \begin{align*}
	    x \backslash^{\star} z&=\negr(\negr z \cdot^{\star} x) &  z /^{\star} x&=\negr(x\cdott \negr z)
	    \end{align*}
	    Then $x \cdot^{\star} y \leq^{\star} z$ iff $\negr z \cdot^{\star} x \leq^{\star} \negr y$ iff $\negrr y \leq^{\star} \negr(\negr z \cdot^{\star} x)$ iff $y \leq^{\star} x\backslash^{\star} z$. 
	    Similarly, $x \cdot^{\star} y \leq^{\star} z$ iff $y \cdott \negr z \leq^{\star} \negr x$ iff $\negrr x \leq^{\star} \negr (y \cdott \negr z)$ iff $x \leq^{\star} z/^{\star} y$. 
	    Therefore, the operations $\backslash^{\star}$ and $/^{\star}$ forms left and right residuals on $A^{\star}$. 
	    Moreover, we set $0=1^{\sim}$. 
	    Then clearly, $\negr x=x \backslash^{\star} 0=0/^{\star} x$ for all $x \in A^{\star}$. 
	    Consequently, $\mathbf{A}^{\star}=(A^{\star},\leq^{\star},\cdot^{\star},\backslash^{\star},/^{\star},1^{\star},0,\top,\bot)$ is a cyclic bounded involutive $r \ell uz$-groupoid. 
	    Obviously, if $\mathbf{A}$ is commutative, then so is $\mathbf{A}^{\star}$, and if $\mathbf{A}$ is square-increasing, then so is $\mathbf{A}^{\star}$; thus $\mathbf{A}^{\star}$ is a member of $\mathsf{CyInRLUG}_{\mathsf{R}}$.
	    For any $x,z \in A$, $x\backslash^{\star} z=\negr(z^{\sim} \cdot^{\star} x)=\negr(x\backslash z)^{\sim}=x \backslash z$ and $z/^{\star}x=\negr(x\cdot^{\star} z^{\sim})=\negr(z/x)^{\sim}=z/x$. 
	    It follows that $\mathbf{A}$ is a subalgebra of the $\mathcal{L}$-reduct of $\mathbf{A}^{\star}$. 
 \end{proof}
Using the above theorem, we prove:
\begin{Cor}
	\label{conserv}
	Let $R \subseteq \{e,c\}$ and $\mathcal{S} \cup \{x \Rightarrow a\}$ be a set of $\mathcal{L}$-sequents. Then
	\[
	\mathcal{S} \vdash_{\mathbf{FNL}_R} x \Rightarrow a \iff \mathcal{S} \vdash_{\mathbf{InFNL}_R} x \Rightarrow a \iff \mathcal{S} \vdash_{\mathbf{CyInFNL}_R} x \Rightarrow a.
	\] 
	That is, $\mathbf{InFNL}_R$ and $\mathbf{CyInFNL}_R$ are strongly conservative over $\mathbf{FNL}_R$.
\end{Cor}
\begin{proof}
	Obviously, if $\mathcal{S} \vdash_{\mathbf{FNL}_R} x \Rightarrow a$, then $\mathcal{S} \vdash_{\mathbf{InFNL}_R} x \Rightarrow a$, and if $\mathcal{S} \vdash_{\mathbf{InFNL}_R} x \Rightarrow a$, then $\mathcal{S} \vdash_{\mathbf{CyInFNL}_R} x \Rightarrow a$. 
	Hence, it suffices to show that if $\mathcal{S} \vdash_{\mathbf{CyInFNL}_R} x \Rightarrow a$ then $\mathcal{S} \vdash_{\mathbf{FNL}_R} x \Rightarrow a$. 
	Suppose that $\mathcal{S} \not\vdash_{\mathbf{FNL}_R} x \Rightarrow a$. 
	By Lemma~\ref{compmodalfree1}, $f(\rho(y)) \leq f(b)$ (for each $y \Rightarrow b \in \mathcal{S}$) and $f(\rho(x)) \not\leq f(a)$ hold, for some $\mathbf{A} \in \mathsf{RLUG}_{\mathsf{R}}$ and some valuation $f$ into $\mathbf{A}$. 
	By Theorem~\ref{inv}, $\mathbf{A}$ is a subalgebra of the $\mathcal{L}$-reduct of a member in $\mathsf{CyInRLUG}_\mathsf{R}$, which we denote by $\mathbf{A}^{\star}$. 
	Define the valuation $v$ into $\mathbf{A}^{\star}$ by $p \mapsto f(p)$ for each propositional variable $p$. 
	Trivially, $f(a)=v(a)$ for any $\mathcal{L}$-formula $a$. 
	Thus we have $v(\rho(y)) \leq v(b)$ and $v(\rho(x)) \not\leq v(a)$ for each $y \Rightarrow b \in \mathcal{S}$; hence $\mathcal{S}\not\models_{\mathsf{CyInRLUG}_{\mathsf{R}}} x \Rightarrow a$.
	By Lemma~\ref{compmodalfree2}, $\mathcal{S} \not\vdash_{\mathbf{CyInFNL}_R} x \Rightarrow a$. 
\end{proof}

As for the extensions of $\mathbf{FNL}$ by the rules of contraction and exchange, the following undecidability result holds:
\begin{Thm}[\cite{Ch15}]
	\label{Chv}
	Let $R \subseteq \{e,c\}$. Given a finite set $\mathcal{S} \cup \{s\}$ of $\mathcal{L}$-sequents, it is undecidable whether $\mathcal{S} \vdash_{\mathbf{FNL}_R} s$, i.e., the finitary consequence relation for $\mathbf{FNL}_R$ is undecidable.
\end{Thm}

On the other hand, Buszkowski had already proved in \cite{Bus16} that the finitary consequence relations for $\mathbf{InFNL}$ and $\mathbf{CyInFNL}$ are both undecidable. 
This result is slightly extended to the following theorem by Corollary~\ref{conserv} and Theorem~\ref{Chv}:
\begin{Thm}
	\label{FCNL}
	Let $R \subseteq \{e,c\}$. The finitary consequence relations for $\mathbf{InFNL}_R$ and $\mathbf{CyInFNL}_R$ are both undecidable.
\end{Thm}

\begin{Lem}
	\label{NACCLL}
    Let $\mathsf{R} \subseteq \{\mathsf{e},\mathsf{c},\mathsf{w}\}$. Every member of $\mathsf{InRLUG}_{\mathsf{R}}$ (resp. $\mathsf{CyInRLUG}_{\mathsf{R}}$) is embedded into the $\mathcal{L}^0$-reduct of a member of $\mathsf{NACCLL}^-_{\mathsf{R}}$ (resp. $\mathsf{NACCLL}_{\mathsf{R}}$).
\end{Lem}
\begin{proof}
	Let $\mathbf{A}$ be a member of $\mathsf{InRLUG}_{\mathsf{R}}$.
    One sees that the Dedekind-MacNeille completion  $\mathbf{F}^{+}_{\mathbf{A}}=(\gamma_{\leq}[{\mathcal{P}}(A)],\cap,\cup_{\gamma_{\leq}},\circ_{\gamma_{\leq}},\backslash,/,\gamma_{\leq}(\{1^{\mathbf{A}}\}),\{0^{\mathbf{A}}\}^{\lhd})$ of $\mathbf{A}$ is also a member of $\mathsf{InRLUG}_{\mathsf{R}}$, and that the map $\{\_\}^{\lhd} \colon A \to \gamma_{\leq}[{\mathcal{P}}(A)]$ is an embedding of $\mathbf{A}$ into $\mathbf{F}^{+}_{\mathbf{A}}$; see \cite[Section 4]{GJ13} for details. 
	We put:
	\[
	S:=\{x \in A \mid x\leq1,x=x\cdot x,\forall a \in A(x\cdot a=a\cdot x),\forall a,b \in A[x\cdot (a\cdot b)=(x\cdot a)\cdot b,(a\cdot b)\cdot x=a\cdot (b\cdot x)]\}.
	\]
	For any $X \in \gamma_{\leq}[\mathcal{P}(A)]$, we put $\oc X=\gamma_{\leq}(X \cap S)$.
	We denote by $\mathbf{D}$ the algebra $\mathbf{F}^{+}_{\mathbf{A}}$ equipped with the new operation $\oc$. 
	It is easy to check that the operation $\oc$ forms a conucleus satisfying the equations $(\oc\mathsf{e})$, $(\oc\mathsf{c})$, $(\oc\mathsf{i})$ $(\oc\mathsf{a1})$, and $(\oc\mathsf{a2})$.
	Hence $\mathbf{D}$ is a member of $\mathsf{NACCLL}^-_{\mathsf{R}}$ and $\mathbf{A}$ is embedded into the $\mathcal{L}^0$-reduct of $\mathbf{D}$. 
	Clearly, if $\mathbf{A}$ is a member of $\mathsf{CyInRLUG}_{\mathsf{R}}$, then so is $\mathbf{F}^+_{\mathbf{A}}$; thus $\mathbf{D}$ is a member of $\mathsf{NACCLL}_{\mathsf{R}}$ and $\mathbf{A}$ is embeddable into the $\mathcal{L}^0$-reduct of $\mathbf{D}$.
\end{proof}

Using the same argument as that in the proof of Corollary~\ref{conserv}, we have:
\begin{Cor}
	\label{conserv2}
	Let $R \subseteq \{e,c,w\}$. $\mathbf{NACCLL}^-_R$ (resp. $\mathbf{NACCLL}_R$) is strongly conservative over $\mathbf{InFNL}_R$ (resp. $\mathbf{CyInFNL}_R$).
\end{Cor}

Define the map $\theta \colon Fm_{\mathcal{L}^0_{\oc}} \cup \{\epsilon\} \to Fm_{\mathcal{L}^0_{\oc}}$ by $\epsilon \mapsto 0$ and $a \mapsto a$ for each $a \in Fm_{\mathcal{L}^0_{\oc}}$.
Given an $\mathcal{L}^0_{\oc}$-sequent $s=x\Rightarrow \delta$, we set $\tau(s)=\oc(\rho(x)\backslash\theta(\delta))$. 

\begin{Lem}
	\label{translation}
	Let $R \subseteq \{e,c,w\}$ and $\{s_1,\ldots,s_n\} \cup \{x \Rightarrow \delta\}$ be a finite set of $\mathcal{L}^0_{\oc}$-sequents. Then, 
	\[
	\{s_1,\ldots,s_n\} \vdash_{\mathbf{NACCLL}^-_R} x \Rightarrow \delta \Longleftrightarrow \vdash_{\mathbf{NACCLL}^-_R} x \circ (\tau(s_1) \circ \cdots(\tau(s_{n-1}) \circ \tau(s_n))\cdots)\Rightarrow \delta,
	\]
    and
    \[
    \{s_1,\ldots,s_n\} \vdash_{\mathbf{NACCLL}_R} x \Rightarrow \delta \Longleftrightarrow \vdash_{\mathbf{NACCLL}_R} x \circ (\tau(s_1) \circ \cdots(\tau(s_{n-1}) \circ \tau(s_n))\cdots) \Rightarrow \delta.
    \]
\end{Lem}
\begin{proof}
	We sketch the proof of the first equivalence. 
	(The second equivalence is shown in the same manner.)
	The \emph{left-to-right} direction is shown by induction on proofs of $x \Rightarrow \delta$ in $\mathbf{NACCLL}^-_R$ from $\{s_1,\ldots,s_n\}$. 
	For the \emph{right-to-left} direction, check that $\{s_1,\ldots,s_n\} \vdash_{\mathbf{NACCLL}^-_R} \varepsilon \Rightarrow \tau(s_i)$ for each $i \in \{1,\ldots,n\}$. 
	Since $\{s_1,\ldots,s_n\} \vdash_{\mathbf{NACCLL}^-_R} x \circ (\tau(s_1) \circ \cdots(\tau(s_{n-1}) \circ \tau(s_n))\cdots) \Rightarrow \delta$, we have $\{s_1,\ldots,s_n\} \vdash_{\mathbf{NACCLL}^-_R} x \Rightarrow \delta$, using the cut rule several times. 
\end{proof}

By Corollary~\ref{conserv2} and Lemma~\ref{translation}, we have:
\begin{Cor}
	\label{translation2}
	Let $R \subseteq \{e,c,w\}$ and $\{s_1,\ldots,s_n\} \cup \{x \Rightarrow \delta\}$ be a finite set of $\mathcal{L}^0$-sequents. Then, 
		\[
		\{s_1,\ldots,s_n\} \vdash_{\mathbf{InFNL}_R} x \Rightarrow \delta \Longleftrightarrow \vdash_{\mathbf{NACCLL}^-_R} x \circ (\tau(s_1) \circ \cdots(\tau(s_{n-1}) \circ \tau(s_n))\cdots) \Rightarrow \delta,
		\]
		and
		\[
		\{s_1,\ldots,s_n\} \vdash_{\mathbf{CyInFNL}_R} x \Rightarrow \delta \Longleftrightarrow \vdash_{\mathbf{NACCLL}_R} x \circ (\tau(s_1) \circ \cdots(\tau(s_{n-1}) \circ \tau(s_n))\cdots) \Rightarrow \delta.
		\]
\end{Cor}

Hence by Theorem~\ref{FCNL} and Corollary~\ref{translation2}, we have:
\begin{Thm}
	\label{classicalundecidable}
	Let $R \subseteq \{e,c\}$. The provability problems for $\mathbf{NACCLL}^-_R$ and $\mathbf{NACCLL}_R$ are undecidable.
\end{Thm}

\appendix
\def\thesection{Appendix \Alph{section}}

\section{Proof of Lemma~\ref{appa}}
\label{appendixa}
\let\temp\theLEM
\renewcommand{\theLEM}{\ref{appa}}
\begin{LEM}
	\begin{enumerate}
		\item $(A^{\star},\leq^{\star},\cdot^{\star},1^{\star})$ is a lattice-ordered unital groupoid.
		\item $x \cdot^{\star} y \leq^{\star} z$ iff $\negr z \cdot^{\star} x \leq^{\star} \negr y$ iff $y \cdott \negr z \leq^{\star} \negr x$, for any $x,y,z \in A^{\star}$.
	\end{enumerate}
\end{LEM}
\begin{proof}
Firstly, we show that Statement (1) holds. 
Let $x \in A^{\sim}$, i.e., $x=y^{\sim}$ for some $y \in A$. 
Then, $y^{\sim} \cdot^{\star} 1^{\star}=y^{\sim} \cdot^{\star} 1=(1\backslash y)^{\sim}=y^{\sim}$, and $1^{\star} \cdot^{\star }y^{\sim}=1 \cdot^{\star }y^{\sim}=(y/1)^{\sim}=y^{\sim}$. 
Obviously, $1^{\star} \cdot^{\star} \top=\top \cdot^{\star} 1^{\star}=\top$, $1^{\star} \cdot^{\star} \bot=\bot \cdot^{\star} 1^{\star}=\bot$, and $1^{\star} \cdot^{\star} x=x \cdot^{\star} 1^{\star}=x$, for all $x \in A$. 
Thus $1^{\star}$ is the unit element of $(A^{\star},\cdot^{\star})$. 
Next, we show that the multiplication $\cdot^{\star}$ is compatible with the order $\leq^{\star}$, i.e., $x \leq^{\star} y$ implies $x \cdot^{\star}z \leq^{\star} y \cdot^{\star} z$ and $z \cdot^{\star} x \leq^{\star} z \cdot^{\star} y$, for any $x,y,z \in A^{\star}$. 
\begin{itemize}
	\item If $x=\bot$, then we have $\bot \cdot^{\star} z=\bot\leq^{\star}y \cdot^{\star} z$ and $z \cdot^{\star} \bot=\bot \leq^{\star} z \cdot^{\star} y$. 
	\item If $x \in A$, then obviously $x \not\leq^{\star}\bot$. 
	Hence we consider the remaining possibilities:
	\begin{enumerate}
		\item If $y \in A$, there are the following additional possibilities:
		\begin{enumerate}
			\item If $z \in A$, obviously $x \cdot^{\star} z \leq^{\star} y \cdot^{\star} z$ and $z \cdot^{\star} x \leq^{\star} z \cdot^{\star} y$.
			\item If $z \in A^{\sim}$, i.e., $z=w^{\sim}$ for some $w \in A$, then $x \cdot^{\star} w^{\sim}=(w/x)^{\sim} \leq^{\star} (w/y)^{\sim}=y \cdot^{\star} w^{\sim}$, and $w^{\sim} \cdot^{\star}x = (x \backslash w)^{\sim} \leq^{\star} (y \backslash w)^{\sim}=w^{\sim} \cdot^{\star} y$. 
			\item If $z \in \{\top,\bot\}$, then $x \cdot^{\star} z=z=y \cdot^{\star} z$, and $z \cdot^{\star} x=z=z\cdot^{\star} y$.
		\end{enumerate}
		\item If $y \in A^{\sim}$, i.e., $y=a^{\sim}$ for some $a \in A$, there are the following additional possibilities: 
		\begin{enumerate}
			\item If $z \in A$, then $x \cdot^{\star} z =x \cdot z<^{\star} (z\backslash a)^{\sim}=a^{\sim} \cdot^{\star} z$ and $z \cdot^{\star} x=z \cdot x <^{\star} (a/z)^{\sim}=z \cdot^{\star} a^{\sim}$. 
			\item If $z \in A^{\sim}$, i.e., $z=b^{\sim}$ for some $b \in A$, then $x \cdot^{\star} b^{\sim}=(b/x)^{\sim}<^{\star}\top=a^{\sim} \cdot^{\star} b^{\sim}$ and $b^{\sim} \cdot^{\star} x=(x \backslash b)^{\sim}<^{\star}\top=b^{\sim} \cdot^{\star} a^{\sim}$.
			\item If $z \in \{\top,\bot\}$, then $x \cdot^{\star} z=z=a^{\sim} \cdot^{\star} z$ and $z \cdot^{\star} x=z=z \cdot ^{\star} a^{\sim}$. 	
		\end{enumerate}
		\item If $y=\top$, there are the following two possibilities:
		\begin{enumerate}
			\item If $z=\bot$, then $x \cdot^{\star} \bot=\bot=\top \cdot^{\star} \bot$ and $\bot \cdot^{\star} x=\bot=\bot \cdot^{\star} \top$. 
			\item If $z\not=\bot$, then $x \cdot^{\star} z \leq^{\star} \top=\top \cdot^{\star} z$ and $z \cdot^{\star} x \leq^{\star} \top=z \cdot^{\star} \top$.
		\end{enumerate} 
	\end{enumerate}
	\item If $x \in A^{\sim}$, i.e., $x=a^{\sim}$ for some $a \in A$, clearly $x \not\leq^{\star} y$ for $y \in A \cup \{\bot\}$. Thus it suffices to consider the cases where $y \in A^{\sim} \cup \{\top\}$. 
	\begin{enumerate}
		\item If $y \in A^{\sim}$, i.e., $y=b^{\sim}$ for some $b \in A$, there are the following additional possibilities:
		\begin{enumerate}
			\item If $z \in A$, then $a^{\sim} \cdot^{\star} z=(z\backslash a)^{\sim} \leq^{\star} (z\backslash b)^{\sim}=b^{\sim} \cdot^{\star} z$ and $z \cdot^{\star} a^{\sim}=(a/z)^{\sim} \leq^{\star} (b/z)^{\sim}=z \cdot^{\star} b^{\sim}$, since $b \leq a$. 
			\item If $z \in A^{\sim} \cup \{\top\}$, then $a^{\sim} \cdot^{\star} z=\top=b^{\sim} \cdot^{\star} z$ and $z \cdot^{\star} a^{\sim}=\top=z \cdot^{\star} b^{\sim}$.
			\item If $z=\bot$, then $a^{\sim} \cdot^{\star} \bot=\bot=b^{\sim} \cdot^{\star} \bot$ and $\bot \cdot^{\star} a^{\sim}=\bot=\bot \cdot^{\star} b^{\sim}$.
		\end{enumerate}
		\item If $y=\top$, there are two additional possibilities:
		\begin{enumerate}
			\item If $z=\bot$, then $a^{\sim} \cdot^{\star} \bot=\bot=\top \cdot^{\star} \bot$ and $\bot \cdot^{\star} a^{\sim}=\bot=\bot \cdot^{\star} \top$. 
			\item If $z\not=\bot$, then $a^{\sim} \cdot^{\star} z \leq^{\star}\top=\top \cdot^{\star} z$ and $z \cdot^{\star} a^{\sim} \leq^{\star} \top=z \cdot^{\star} \top$.
		\end{enumerate}
	\end{enumerate}
	\item If $x=\top$, then it suffices to consider the case where $y=\top$. This case is trivial.
\end{itemize}
Thus $(A^{\star},\leq^{\star},\cdot^{\star},1^{\star})$ is a partially-ordered unital groupoid. 
To show that $(A^{\star},\leq^{\star},\cdot^{\star},1^{\star})$ is lattice-ordered, it suffices to show that $\inf\{a^{\sim},b^{\sim}\}$ and $\sup\{a^{\sim},b^{\sim}\}$ exist in $A^{\star}$ for any $a,b \in A$. 
We show that $\inf\{a^{\sim},b^{\sim}\}=(a\vee b)^{\sim}$. 
The element $(a \vee b)^{\sim}$ is a lower bound of $\{a^{\sim},b^{\sim}\}$, because $a \leq a \vee b$ iff $(a \lor b)^{\sim} \leq^{\star} a^{\sim}$ and $b \leq a \vee b$ iff $(a \lor b)^{\sim} \leq^{\star} b^{\sim}$. 
Let $x$ be a lower bound of $\{a^{\sim},b^{\sim}\}$; i.e., $x \leq^{\star} a^{\sim}$ and $x \leq^{\star} b^{\sim}$. 
We show that $x \leq^{\star} (a \vee b)^{\sim}$ by case analysis.
If $x \in A \cup \{\bot\}$, it is clear that $x <^{\star} (a \vee b)^{\sim}$. 
If $x \in A^{\sim}$, i.e., $x=y^{\sim}$ for some $y \in A$, then we have $a \vee b \leq y$, since $y^{\sim} \leq^{\star} a^{\sim}$ iff $a \leq y$ and $y^{\sim} \leq^{\star} b^{\sim}$ iff $b \leq y$.
Hence $y^{\sim} \leq^{\star} (a \vee b)^{\sim}$. 
Obviously, it is impossible that $x=\top$. 
Thus we have $\inf\{a^{\sim},b^{\sim}\}=(a\vee b)^{\sim}$. 
One checks $\sup\{a^{\sim},b^{\sim}\}=(a \wedge b)^{\sim}$ in a similar way.

To confirm that Statement (2) holds, clearly it suffices to show that $x \cdot^{\star} y \leq^{\star} z$ iff $\negr z \cdot^{\star} x \leq^{\star} \negr y$ holds for any $x,y,z \in A^{\star}$. 
Again, we perform a case-by-case analysis.
\begin{itemize}
	\item If $x=\bot$, then we have $\bot \cdot^{\star} y = \bot \leq^{\star} z$ and $\negr z \cdot^{\star} \bot =\bot \leq^{\star} \negr y$.
	\item If $x=\top$, we consider the following two possibilities:
	\begin{enumerate}
		\item If $y=\bot$, then we have $\top \cdot^{\star} \bot=\bot \leq^{\star} z$ and $\negr z \cdot^{\star} \top \leq^{\star} \top=\negr \bot$.
		\item If $y\not=\bot$, we consider the following two cases:
		\begin{enumerate}
			\item If $z\not=\top$, then $\top \cdot^{\star} y=\top \not \leq^{\star} z$ and $\negr z \cdot^{\star} \top=\top \not\leq^{\star} \negr y$. 
			\item If $z=\top$, then $\top \cdot^{\star} y \leq^{\star} \top$ and $\negr \top \cdot^{\star} \top=\bot \leq^{\star} \negr y$.
		\end{enumerate} 
	\end{enumerate} 
	\item If $x \in A$, then we consider the following possibilities:
	\begin{enumerate}
	    \item If $y \in A$, there are the following additional possibilities:
		\begin{enumerate}
			\item If $z \in A$, then $x \cdot^{\star} y \leq^{\star} z$ iff $x \cdot y \leq z$ iff $y \leq x \backslash z$ iff $(x \backslash z)^{\sim} \leq^{\star} y^{\sim}$ iff $z^{\sim} \cdot^{\star} x \leq^{\star} y^{\sim}$ iff $\negr z \cdot^{\star} x \leq^{\star} \negr y$. 
			\item If $z \in A^{\sim}$, i.e., $z=a^{\sim}$ for some $a \in A$, then $x \cdot^{\star} y =x \cdot y<^{\star} a^{\sim}$ and $\negr a^{\sim} \cdot^{\star} x=a \cdot^{\star} x <^{\star} y^{\sim}=\negr y$. 
			\item If $z=\top$, then $x \cdot^{\star}y=x \cdot y <^{\star} \top$ and $\negr \top \cdot^{\star} x=\bot<^{\star} y^{\sim}=\negr y$.
			\item If $z=\bot$, then $x \cdot^{\star} y=x \cdot y \not\leq^{\star} \bot$ and $\negr \bot \cdot^{\star} x=\top \not\leq^{\star} y^{\sim}=\negr y$.
		\end{enumerate}
	    \item If $y \in A^{\sim}$, i.e., $y=a^{\sim}$ for some $a \in A$, there are the following possibilities:
	    \begin{enumerate}
	    	\item If $z \in A \cup \{\bot\}$, then $x \cdot^{\star} a^{\sim}=(a/x)^{\sim} \not\leq^{\star} z$ and $\negr z \cdot^{\star} x \not\leq^{\star} a=\negr a^{\sim}$.
	    	\item If $z \in A^{\sim}$, i.e., $z=b^{\sim}$ for some $b \in A$, then $x \cdot^{\star} a^{\sim} \leq^{\star} b^{\sim}$ iff $(a/x)^{\sim} \leq^{\star} b^{\sim}$ iff $b \leq a/x $ iff $b \cdot x \leq a$ iff $b \cdot^{\star} x \leq^{\star} a$ iff $\negr b^{\sim} \cdot^{\star} x \leq^{\star} \negr a^{\sim}$. 
	    	\item If $z=\top$, we have $x \cdot^{\star} a^{\sim}=(a/x)^{\sim} <^{\star} \top$ and $\negr \top \cdot^{\star} x=\bot <^{\star} a=\negr a^{\sim}$. 
	    \end{enumerate}
        \item If $y=\bot$, then we have $x \cdot^{\star} \bot=\bot \leq^{\star} z$ and $\negr z \cdot^{\star} x \leq^{\star} \top=\negr \bot$. 
        \item If $y=\top$, then we consider the following two cases: 
        \begin{enumerate}
        	\item If $z\not=\top$, then $x \cdot^{\star} \top=\top \not\leq^{\star} z$ and $\negr z \cdot^{\star} x \not\leq^{\star} \bot=\negr \top$. 
        	\item If $z=\top$, then $x \cdot^{\star} \top \leq^{\star} \top$ and $\negr \top \cdot^{\star} x=\bot \leq^{\star} \negr \top$.  
        \end{enumerate}
	\end{enumerate}
\item If $x \in A^{\sim}$, i.e., $x=a^{\sim}$ for some $a \in A$, then we consider the following cases:
\begin{enumerate}
	\item If $y \in A$, then there are the following possibilities:
	\begin{enumerate}
		\item If $z \in A \cup \{\bot\}$, then $a^{\sim} \cdot^{\star} y=(y\backslash a)^{\sim} \not\leq^{\star} z$ and $\negr z \cdot^{\star} a^{\sim}=\top \not\leq^{\star} y^{\sim}=\negr y$.
		\item If $z \in A^{\sim}$, i.e., $z=b^{\sim}$ for some $b \in A$, then $a^{\sim} \cdot^{\star} y \leq^{\star} b^{\sim}$ iff $(y \backslash a)^{\sim} \leq^{\star} b^{\sim}$ iff $b \leq y \backslash a$ iff $y \cdot b \leq a$ iff $y \leq a/b$ iff $(a/b)^{\sim} \leq^{\star} y^{\sim}$ iff $b \cdot^{\star} a^{\sim} \leq^{\star} y^{\sim}$ iff $\negr b^{\sim} \cdot^{\star} a^{\sim} \leq^{\star} \negr y$. 
		\item If $z=\top$, then we have $a^{\sim} \cdot^{\star} y=(y\backslash a)^{\sim}<^{\star} \top$ and $\negr \top \cdot^{\star} a^{\sim}=\bot <^{\star} y^{\sim}=\negr y$.
	\end{enumerate}
    \item If $y \in A^{\sim}$, i.e., $y=b^{\sim}$ for some $b \in A$, then we consider the following two possibilities: 
    \begin{enumerate}
    	\item If $z \not=\top$, then we have $a^{\sim} \cdot^{\star} b^{\sim}=\top \not\leq^{\star} z$, and $\negr z \cdot^{\star} a^{\sim} \not\leq^{\star} \negr b^{\sim}$, since if $z \in A^{\sim}$, i.e., $z=c^{\sim}$ for some $c \in A$, then $\negr z \cdot^{\star} a^{\sim}=c \cdot^{\star} a^{\sim}=(a/c)^{\sim}\not\leq^{\star} b=\negr b^{\sim}$, and if $z \in A \cup \{\bot\}$, then $\negr z \cdot^{\star} a^{\sim}=\top \not\leq^{\star} b=\negr b^{\sim}$.
    	\item If $z=\top$, then we have $a^{\sim} \cdot^{\star} b^{\sim}=\top \leq^{\star} \top$ and $\negr \top \cdot^{\star} a^{\sim}=\bot<^{\star} b=\negr b^{\sim}$. 
    \end{enumerate}
    \item If $y=\top$, there are the following two possibilities:
    \begin{enumerate}
    	\item If $z\not=\top$, then $a^{\sim} \cdot^{\star} \top=\top \not\leq^{\star} z$ and clearly $\negr z \cdot^{\star} a^{\sim} \not\leq^{\star} \bot=\negr \top$, since if $z \in A \cup \{\bot\}$ then $\negr z \cdot^{\star} a^{\sim}=\top \not\leq^{\star} \bot=\negr \top$ and, if $z \in A^{\sim}$, i.e., $z=b^{\sim}$ for some $b \in A$, then $\negr z \cdot^{\star} a^{\sim}=b \cdot^{\star} a^{\sim}=(a/b)^{\sim} \not\leq^{\star} \bot=\negr \top$. 
    	\item If $z = \top$, then obviously $a^{\sim} \cdot^{\star} \top \leq^{\star} \top$ and $\negr \top \cdot^{\star} a^{\sim}=\bot \leq^{\star} \negr \top$. 
    \end{enumerate} 
    \item If $y=\bot$, then we have $a^{\sim} \cdot^{\star} \bot=\bot \leq^{\star} z$ and $\negr z \cdot^{\star} a^{\sim} \leq^{\star} \top=\negr \bot$.
\end{enumerate}
\end{itemize}
\end{proof}

\end{document}